\documentclass[11pt]{article}
\usepackage{amssymb,amsmath,amsthm,bm}
\usepackage[colorlinks=true, urlcolor=blue,linkcolor=blue, citecolor=blue]{hyperref}
\usepackage{caption}
\usepackage{graphicx, float}
\usepackage{color, mathtools, tikz, xcolor}
\usepackage{enumitem}
	\setlist[enumerate]{noitemsep,nolistsep}
\usepackage[labelformat=simple]{subcaption}

\usetikzlibrary{decorations.pathreplacing,arrows,calc}
\usetikzlibrary{shapes.geometric}
\usetikzlibrary{decorations.markings}
\usepackage{bbm}
\usepackage[mathlines]{lineno}

\usepackage[UKenglish]{babel}
\usepackage[UKenglish]{isodate}

\usepackage{geometry}
\numberwithin{equation}{section}

\numberwithin{figure}{section}

	\newtheorem{theorem}{Theorem}[section]
	
	\newtheorem{conjecture}[theorem]{Conjecture}
	
	\newtheorem{claim}[theorem]{Claim}
	\newtheorem{proposition}[theorem]{Proposition}
	\newtheorem{lemma}[theorem]{Lemma}

\newenvironment{proofclaim}[1][Proof of claim]{\begin{proof}[#1]}{\end{proof}}

\def\al#1{}
	\renewcommand{\al}[1]{\footnote{\textbf{AL: }#1}}

\newcommand\pd{\bar{\delta}^0}

\date{\today}

\begin{document}

\tikzset{->-/.style={
		decoration={markings, mark=at position 0.6 with {\arrow{latex}}},postaction={decorate}}}

\title{Long antipaths in oriented graphs}

\author{Yuping Gao\footnote{School of Mathematics and Statistics, Lanzhou University, Lanzhou 730000, China.
			Email: {\tt gaoyp@lzu.edu.cn}.
			Supported in part by  NSFC grant, No. 12271228 and China Scholarship Council, No. 202406180027.}
		\qquad
		Allan Lo\footnote{School of Mathematics, University of Birmingham, Birmingham, B15 2TT, UK. Email: {\tt s.a.lo@bham.ac.uk}.}
	}

\maketitle
\begin{abstract}
An antidirected path is an oriented path in which every vertex sees either just incoming or just outgoing edges.
We prove that every oriented graph with minimum semidegree at least $k$ contains an antidirected path of length $2 k -1$.
This confirms a conjecture of Stein.
	
\medskip

\end{abstract}

\section{Introduction}

A renowned result of Dirac~\cite{D1952} states that every graph~$G$ on $n \ge 3$ vertices with $\delta(G) \ge n/2$ contains a Hamilton cycle.
This problem naturally extended to oriented graphs and digraphs with large minimum semidegree.
An \emph{oriented graph} is a loopless directed graph without cycle of length~$2$.
The \emph{minimum semidegree~$\delta^0(G)$ of a digraph~$G$} is the minimum value amongst the in-degree and out-degree of all its vertices.
Keevash, K\"uhn and Osthus~\cite{MR2472138} proved that every large oriented graph~$G$ on $n$ vertices with $\delta^0(G) \ge (3n-4)/8$ contains a directed Hamilton cycle.
See \cite{MR2889513} for a survey on Hamilton cycles in digraphs and oriented graphs.

What is the longest oriented path in an oriented graph?
There are two natural oriented paths, directed and antidirected.
A \emph{directed path} is an oriented path in which all its edges are directed in the same direction.
On the other hand, an \emph{antidirected path}, or an \emph{antipath} for short,  is an oriented path in which each of its vertices has in-degree~$0$ or out-degree~$0$.
Observe that a directed path does not contain an antipath of length~$2$ and vice versa.
The \emph{length} of a path is the number of edges it contains.

Oriented paths in \emph{tournaments} (oriented complete graphs) are well studied.
A classical result of R\'edei~\cite{Redei} states that every tournament contains a directed Hamilton path.
Gr\"unbaum~\cite{Grunbaum} proved that all but three tournaments contain a Hamilton antipath.
Rosenfeld~\cite{Rosenfeld} later conjectured that every tournament on at least $8$ vertices contains all oriented Hamilton paths.
Thomason~\cite{Thomason} proved the conjecture holds for tournaments on at least $2^{128}$ vertices.
The conjecture was finally settled by Havet and Thomass\'e~\cite{HavetThomasse}.

For oriented graphs~$G$, Jackson~\cite{J1981} proved that there is a directed path of length~$2\delta^{0}(G)$ or a directed Hamilton cycle.
Motivated by this result, Stein~\cite{S2020} proposed the following minimum semidegree condition forcing an oriented graph to contain any given orientation of a path, which can also be viewed as an analogue of Rosenfeld's conjecture for oriented graphs.

\begin{conjecture}[Stein~\cite{S2020}]\label{conj:Stein}
Every oriented graph~$G$ contains all oriented paths of length~$2\delta^{0}(G)-1$.
\end{conjecture}

If the conjecture holds, then it is best possible by considering (disjoint union of) any regular tournaments on $2\delta^{0}(G)-1$ vertices.
There are various results supporting this conjecture.
The directed path case is proved by Jackson~\cite{J1981}.
Chen, Hou, and Zhou~\cite{CHZ2025-EJC} proved the conjecture for oriented paths with one change in direction.
For arbitrary orientations, Kelly~\cite{Kelly} proved the case when $\delta^0(G) \ge 3n/8+ o(n)$ with $n$ large, and Stein and Trujillo-Negrete~\cite{MR5089877} proved the case when $G$ contains no oriented $4$-cycles.

In this paper, we investigate Conjecture~\ref{conj:Stein} for antipaths.
The \emph{minimum pseudo-semidegree~$\pd(G)$} of a non-empty digraph~$G$ is the minimum value amongst all non-zero in-degree and out-degree of all its vertices.
In other words, the minimum pseudo-semidegree is the maximum $k$ such that each vertex has either out-degree~$0$ or out-degree at least~$k$, and has either in-degree~$0$ or in-degree at least~$k$.
If $G$ is empty, then $\pd(G) = 0$.
Clearly, $\pd(G) \ge \delta^0(G) $.
Stein~\cite{S2023} proposed the following conjecture in terms of the minimum pseudo-semidegree condition.

\begin{conjecture}[Stein~\cite{S2023}]\label{conj:pseudo}
Every oriented graph~$G$ with $\pd(G) \ge k$ contains an antipath of length~$2k-1$.
\end{conjecture}

Note that up to isomorphism there is only one antipath of odd length (and two for even length).
Clearly, Conjecture~\ref{conj:pseudo} implies Conjecture~\ref{conj:Stein} for antipaths.
In fact, the converse is true, by an observation of Stein and Z\'{a}rate-Guer\'{e}n (c.f.~\cite[Proof of Lemma 7.2]{SZ2024}).
It is easy to see that Conjecture~\ref{conj:pseudo} holds for $k \le 2$.
When $\delta^0(G)$ is linear in~$n$, Stein and Z\'{a}rate-Guer\'{e}n~\cite{SZ2024} showed that Conjecture~\ref{conj:pseudo} holds asymptotically, that is, when $k = \eta n$ and $\delta^0(G) \ge k + \varepsilon n $.
Klimo\v{s}ov\'{a} and Stein~\cite{KS2023} proved that Conjecture~\ref{conj:pseudo} holds with $\pd(G) \ge (6k-5)/4$ instead.
The lower bound of $\pd(G)$ was later improved by Chen, Hou, and Zhou~\cite{CHZ2025} to $(4k-1)/3$, and by Skokan and Tyomkyn~\cite{ST2025} to $(10k-5)/8$.
Recently, Grzesik and Skrzypczyk~\cite{GS2025} further improved the bound to $\pd(G) > k-1+\sqrt{2k-4}/2$.

In this paper, we confirm Conjecture~\ref{conj:pseudo} (and so Conjecture~\ref{conj:Stein}  for antipaths).

\begin{theorem}\label{thm:pseudo}
Every oriented graph~$G$ with $\pd(G) \ge k$ contains an antipath of length~$2k-1$.
\end{theorem}

The next proposition shows that Theorem~\ref{thm:pseudo} is best possible as one cannot seek an antipath nor anticycle of length~$2k$.
Here, an \emph{anticycle} is a cycle (of even length) in which each of its vertices has in-degree~$0$ or out-degree~$0$.

\begin{proposition}
For each $k \in \mathbb{N}$, there exists an oriented graph~$G$ with $\delta^0(G) = \pd(G) =k$ without an antipath or anticycle of length~$2k$.
\end{proposition}

\begin{proof}
For $k = 1$ and $k=2$, let $G$ be a directed cycle and a regular tournament on $5$ vertices, respectively.
It is easy to check that $G$ is the desired oriented graph.

Now suppose that $k\ge 3$.
Let $T$ be a tournament on $k$ vertices with $\delta^0(T) = \lfloor (k-1)/2 \rfloor$.
(If $k$ is odd, then $T$ is a regular tournament.
If $k$ is even, then $T$ can be obtained by deleting one vertex from a regular tournament on $k+1$ vertices.)
We construct an oriented graph~$G$ on~$2\lceil (3k+1)/2 \rceil$ vertices as follows.
The vertex set can be partitioned into $V_1,V_2,V_3,V_4$ of sizes $\lceil (k+1)/2 \rceil,k,\lceil (k+1)/2 \rceil,k$, respectively.
Each of~$V_2$ and~$V_4$ induces a copy of~$T$.
Also, $G$ contains all edges from~$V_i$ to~$V_{i+1}$, where we take $V_5=V_1$, and no other edges.
Note that $\delta^0(G)=\pd(G) = \min \{ k , \lfloor (k-1)/2 \rfloor +\lceil (k+1)/2 \rceil \}= k$.

Suppose to the contrary that $G$ contains an antipath or anticycle of length~$2k$.
Since any oriented path between $V_2$ and $V_4$ contains a directed path of length~$2$, we may assume without loss of generality that such a copy is contained entirely in~$V_1 \cup V_2 \cup V_3$.
First suppose that $G$ contains an antipath~$P$ of length~$2k$.
If $k$ is odd, then $V(P) = V_1 \cup V_2 \cup V_3$.
Otherwise, $V(P)$ contains all but one vertex in~$V_1 \cup V_2 \cup V_3$.
In both cases, we have $|V(P) \cap (V_1 \cup V_3)| \ge |V(P) \cap V_2|$ and that $P$ contains at least one vertex from $V_1$ and one from~$V_3$.
Since $V_1 \cup V_3$ is an independent set, $P$ must alternate between $V_1 \cup V_3$ and~$V_2$.
Hence, there exist three consecutive vertices in~$P$ each from a distinct~$V_i$, which induce in~$P$ a directed path of length~$2$, a contradiction.
The case when $G$ contains an anticycle of length~$2k$ is proved similarly.
\end{proof}

\subsection{Notation}

For integers $p$ and $q$, let  $[p,q] = \{i \in \mathbb{Z} \colon p \le i \le q \}$. In particular, we write $[q]$ for $[1,q]$.

Let $G$ be a digraph.
For two vertices $u,v\in V(G)$, we write $uv\in E(G)$ if the edge is directed from $u$ to~$v$, and in this case we call $v$ an \emph{out-neighbour} of~$u$ and $u$ an \emph{in-neighbour} of~$v$.
The \emph{in-neighbourhood} and~\emph{out-neighbourhood} of~$v$, denoted by $N^-(v)$ and $N^+(v)$, are the sets of in-neighbours and out-neighbours of~$v$, respectively.
The \emph{in-degree} of~$v$, denoted by~$d^-(v)$, is the number of in-neighbours of~$v$, while the \emph{out-degree} of~$v$, denoted by~$d^+(v)$, is the number of out-neighbours of~$v$.

For a subset~$X\subseteq V(G)$, we write $N^+(X)=\bigcup_{x \in X} N^+(x)$.
For disjoint $U,V \subseteq V(G)$, let $e(U,V)$ denote the number of edges from~$U$ to~$V$.
We simply write $v$ for $\{v\}$.
For $v \in V(G)$ and $U\subseteq V(G)$, we write $N^+(v,U) = N^+(v) \cap U$ and $d^+(v,U) = |N^+(v,U)|$, and we define $N^-(v,U)$ and $d^-(v,U)$ analogously.
The \emph{reverse of~$G$} is the oriented graph obtained from $G$ by reversing all its edge orientations.

\subsection{Proof sketch  of Theorem~\ref{thm:pseudo}}

We now sketch the proof of Theorem~\ref{thm:pseudo}.
Let $G$ be an oriented graph with $\pd(G) \ge k$ with no antipath of length~$2k-1$.
Let $P = v_0 v_1 \dots v_m$ be a longest antipath in~$G$ with $v_0v_1 \in E(G)$.
By a result of Klimo\v{s}ov\'{a} and Stein~\cite{KS2023} (see Lemma~\ref{lem:oddlength}), we have $m = 2\ell +1$.
We now consider a particular ``blow-up'' of~$P$ by replacing $v_0$ with an independent set~$V_0$,  and adding all edges from~$V_0$ to~$v_1$ (see Section~\ref{sec:lemmas} for the formal definition).
This notion of antipath blow-up was already considered by Grzesik and Skrzypczyk~\cite{GS2025}.
More importantly, they showed that if $|V_0| \ge 2$, then  $e(V_0, \{ v_{2i},v_{2i+1},v_{2i+2},v_{2i+3}\} ) \le 4|V_0|+1$ for all $i \in [\ell-2]$.
This played a crucial role in their proof.
In this paper, we show that one can also replace~$v_{2\ell+1}$ with an independent set~$V_{2\ell+1}$ of size~$|V_{2\ell+1}| \ge 2$ and add all edges from~$v_{2\ell}$ to~$V_{2\ell+1}$.
Furthermore, we then have ``on average''
\begin{align*}
	e(V_0, \{ v_{2i},v_{2i+1} \}) + e( \{ v_{2i},v_{2i+1} \}, V_{2\ell+1} ) \le |V_0|+|V_{2\ell+1}|
\end{align*}
for all $i \in [\ell-1]$ (see Claim~\ref{clm:surplus}).
This then gives us an upper bound on  $e(V_0, \{v_j \colon j \in [2\ell]\}) + e( \{v_j \colon j \in [2\ell]\}, V_{2\ell+1} )$, whilst $\pd(G)$ provides a lower bound.
The most involving part of our proof is actually finding such an antipath blow-up with $|V_0|,|V_{2\ell+1}| \ge 2$ (see Lemma~\ref{lem:twofatends} and it is proved in Section~\ref{sec:lemmaproof}).

\section{Auxiliary lemmas and proof of Theorem~\ref{thm:pseudo}}\label{sec:lemmas}

To prove our main result, we introduce the notion of an antipath blow-up, which generalizes an antipath.
First we state two lemmas of Klimo\v{s}ov\'{a} and Stein~\cite{KS2023}.

\begin{lemma}[Klimo\v{s}ov\'{a} and Stein~\cite{KS2023}]\label{lem:oddlength}
Let $k \in \mathbb{N}$ and let $G$ be an oriented graph with $\pd(G) \ge k$.
Suppose that a longest antipath in~$G$ has length~$m < 2k-1$.
Then $m$ is odd.
\end{lemma}

\begin{lemma}[Klimo\v{s}ov\'{a} and Stein~\cite{KS2023}]\label{lem:cycle}
Let $k \in \mathbb{N}$ and let $G$ be an oriented graph with $\pd(G) \ge k$.
Suppose that $G$ has an anticycle of length $2\ell+2< 2k-1$.
Then $G$ contains an antipath of length~$2\ell+2$.
\end{lemma}

We now introduce antipath blow-ups. The purpose of this notion is to encode a family of antipaths that share the same internal vertices but may have several possible choices at one or both ends. This allows us to exploit degree conditions at the endpoints and to compare different extensions of a longest antipath.
Let $G$ be an oriented graph and let $\ell \in \mathbb{N} \cup \{0\}$.
For disjoint sets $V_0, V_1, \dots, V_{2\ell+1} \subseteq V(G)$, we say that $V_0 V_1 \dots V_{2 \ell +1}$ is an \emph{antipath blow-up of length $2 \ell +1$} if for all $i \in [0,2\ell+1] $ and all $v_i \in V_i$, $v_0 v_1 \dots v_{2\ell+1}$ is an antipath in~$G$ with $v_0 v_1 \in E(G)$.
Thus all edges of this antipath are directed from vertices with even index to vertices with odd index.
Throughout this paper, we shall only use antipath blow-ups in which $V_i=\{v_i\}$ for all $i\in[2\ell]$.
Therefore, we write $V_0v_1\dots v_{2\ell}V_{2\ell+1}$ instead of $V_0V_1\dots V_{2\ell+1}$, and we write
$V_0v_1\dots v_{2\ell}v_{2\ell+1}$ when $V_{2\ell+1}=\{v_{2\ell+1}\}$.

For $i\in[\ell-1]$, we say that $i$ is \emph{$V_0$-surplus} if $e\bigl(V_0,\{v_{2i},v_{2i+1}\}\bigr) \ge |V_0|+1$.
Similarly, we say that $i$ is \emph{$V_{2\ell+1}$-surplus} if $e\bigl(\{v_{2i},v_{2i+1}\},V_{2\ell+1}\bigr) \ge |V_{2\ell+1}|+1$.

In the next lemma, we study various properties of an antipath blow-up.
It should be noted that a slightly weaker version of~\ref{itm:property:b} was already proved by Grzesik and Skrzypczyk~\cite[Claim~12]{GS2025}.

\begin{lemma} \label{lem:property}
Let $k \in \mathbb{N}$ and let $G$ be an oriented graph with $\pd(G) \ge k$.
Suppose that a longest antipath in~$G$ has length~$2\ell+1 < 2k-1$.
Let $V_0v_1\dots v_{2\ell+1}$ be an antipath blow-up in~$G$ with $|V_0| \ge 2$.
Then the following statements hold:
\begin{enumerate}[label={\rm (\alph*)}]
	\item \label{itm:property:a}
	$N^+(V_0) \subseteq \{v_j \colon j \in [2\ell]\}$;

	\item \label{itm:property:b} if $xv_{2i+1}, x' v_{2i} \in E(G)$ for distinct $x,x' \in V_0$, then $xv_{2i-1},xv_{2i}, x'v_{2i-1} \notin E(G)$;	

	\item \label{itm:property:c}
	if $i$ is $V_0$-surplus, then $d^-(v_{2i},V_0) = 1$, $V_0 \subseteq N^-(v_{2i+1})$, $N^+(v_{2i}) \subseteq \{v_j \colon j \in [2\ell]\}$, and $v_{2i-1} \notin N^+(V_0)$; in particular, $e(V_0,\{ v_{2i}, v_{2i+1}\}) = |V_0| +1$ and $i \in [2, \ell-1]$;

	\item \label{itm:property:d}
	if $i$ is $V_0$-surplus and $e(V_0, \{ v_{2i-2}, v_{2i-1}, v_{2i}, v_{2i+1}\} ) \ge 2|V_0| +1$, then $e(V_0, \{ v_{2i-2}, v_{2i-1}, v_{2i}, v_{2i+1}\} ) = 2|V_0| +1$ and $V_0 \subseteq N^-(v_{2i-2})$;
	
	\item \label{itm:property:e}
	there are no two consecutive integers that are both $V_0$-surplus;

	\item \label{itm:property:g}
	if $i$ is $V_0$-surplus and $v_{2j+1} \in N^+(V_0)$, then $v_{2j} v_{2i} \notin E(G)$;

	\item \label{itm:property:h}
	if $i$ and $i'$ are $V_0$-surplus, then  $v_{2i} v_{2i'} \notin E(G)$;
	
	\item \label{itm:property:i}
	if $i$ is $V_0$-surplus and there are $m$ many $j \in [\ell-1]$ that are $V_0$-surplus, then $ | N^-(v_{2i})  \setminus (V_0 \cup \{v_j \colon j \in [2\ell]\})| \ge m +3$;
	
	\item \label{itm:property:f}
	if $i$ is $V_0$-surplus, then $ v_{2i-2} v_{2\ell+1} , v_{2i} v_{2\ell+1} \notin E(G)$;
	
	\item \label{itm:property:j}
	if $i$ and $i'$ are $V_0$-surplus, $j\in[\ell-1]$,
	and $v_{2j+1}\in N^+(V_0)$, then $v_{2j} \notin N^+(v_{2i}) \cap N^+(v_{2i'})$.
\end{enumerate}
\end{lemma}

By considering the reverse of $G$, the corresponding statements hold for antipath blow-up $v_0v_1v_2\dots v_{2\ell}V_{2\ell+1}$ with $|V_{2\ell+1}|\ge 2$.
For example, Lemma~\ref{lem:property}\ref{itm:property:c} would imply that if $v_0v_1v_2\dots v_{2\ell}V_{2\ell+1}$ is an antipath blow-up with $|V_{2\ell+1}|\ge 2$ and $i$ is $V_{2\ell+1}$-surplus, then $d^+(v_{2i+1},V_{2\ell+1}) = 1$, $V_{2\ell+1} \subseteq N^+(v_{2i})$, $N^-(v_{2i+1}) \subseteq \{v_j \colon j \in [2\ell]\}$, and $v_{2i+2} \notin N^-(V_{2\ell+1})$.

\begin{proof}[Proof of Lemma~\ref{lem:property}]
Let $V' =  \{v_j \colon j \in [2\ell]\}$. Since $2\ell+1<2k-1$ and $2k-1$ is odd, we have $2\ell+2<2k-1$.
By Lemma~\ref{lem:cycle}, $G$ does not contain any anticycle of length $2 \ell +2$.


Clearly, $N^+(V_0) \subseteq \{v_j \colon j \in [2\ell+1]\}$ or else we would obtain a longer antipath, a contradiction.
If $v_{2\ell+1} \in N^+(V_0)$, then $G$ contains an anticycle of length $2\ell +2$, a contradiction.
Thus $N^+(V_0) \subseteq V'$, implying~\ref{itm:property:a}.


Consider distinct $x,x' \in V_0$ with $xv_{2i+1}, x'v_{2i} \in E(G)$.
Suppose to the contrary that~\ref{itm:property:b} is false.
Then $G$ contains an antipath of length~$2\ell+2$, namely
\begin{align*}
	& v_{2i} x' v_1 \dots v_{2i-1} x v_{2i+1} \dots v_{2\ell+1}
		& & \text{if $xv_{2i-1} \in E(G)$,}\\
	& 	v_{2i-1} \dots  v_1 x' v_{2i} x v_{2i+1} \dots v_{2\ell+1}
		& & \text{if $xv_{2i} \in E(G)$,}\\
	& v_{2i} x'v_{2i-1} \dots v_{1} x v_{2i+1} \dots v_{2\ell+1}
		& & \text{if $x' v_{2i-1}\in E(G)$,}
\end{align*}
(see Figure~\ref{fig:pro:2}), a contradiction.
Hence~\ref{itm:property:b} holds.
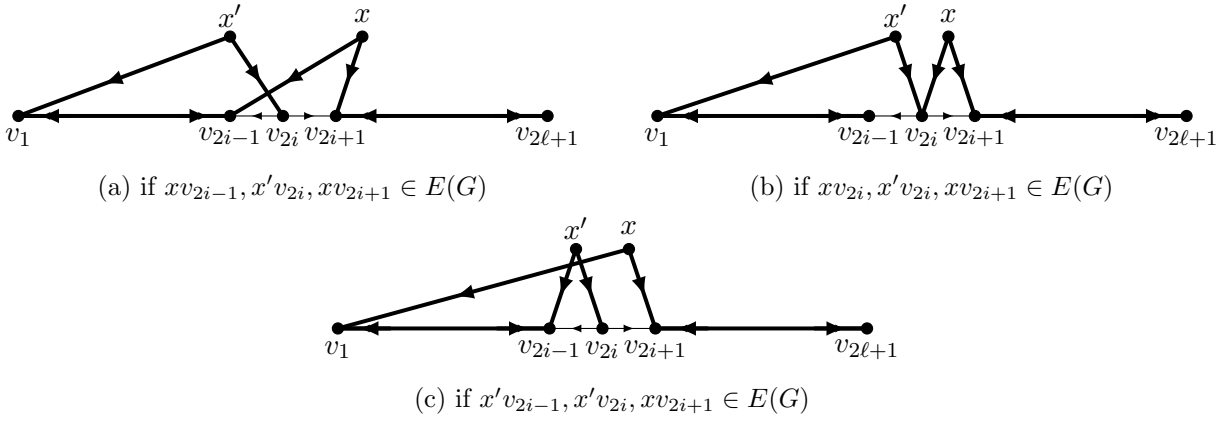
\begin{figure}
	\centering
	\begin{subfigure}{0.475\textwidth}
		\centering
		\begin{tikzpicture}[scale=0.7]

				\foreach \x/\y in {5/x',7.5/x}
					{
					\filldraw[black] (\x,1.5) circle (3pt);
					\node at (\x,1.9) {${\y}$};
					}
			
				\foreach \x/\y in { 1/1,5/2i-1,6/2i,7/2i+1,11/2\ell+1}
					{
					\filldraw[black] (\x,0) circle (3pt);
					\node at (\x,-0.4) {$v_{\y}$};
					}
				\foreach \x in {6,10,4}
					{		\draw[->-,thin] (\x,0)--(\x+1,0);		}
					\foreach \x in {2,6,8}
					{		\draw[->-,thin] (\x,0)--(\x-1,0);		}
					
				\begin{scope}[line width = 1.5pt]
					\draw[->-] (5,1.5)--(6,0);
					\draw[->-] (5,1.5)--(1,0);
					\draw[->-] (7.5,1.5)--(5,0);
					\draw[->-] (7.5,1.5)--(7,0);
					\draw (1,0)--(5,0);
					\draw (7,0)--(11,0);
					\foreach \x/\y in {2/1,4/5,8/7,10/11}
					{		\draw[->-] (\x,0)--(\y,0);		}	
				\end{scope}
	\end{tikzpicture}
	\caption{if $xv_{2i-1}, x'v_{2i}, xv_{2i+1} \in E(G)$}
	\label{fig:pro:2:1}
	\end{subfigure}
	\begin{subfigure}{0.475\textwidth}
		\centering
		\begin{tikzpicture}[scale=0.7]

				\foreach \x/\y in { 5.5/x',6.5/x}
					{
					\filldraw[black] (\x,1.5) circle (3pt);
					\node at (\x,1.9) {${\y}$};
					}
			
				\foreach \x/\y in { 1/1,5/2i-1,6/2i,7/2i+1,11/2\ell+1}
					{
					\filldraw[black] (\x,0) circle (3pt);
					\node at (\x,-0.4) {$v_{\y}$};
					}
				\foreach \x in {6,10,4}
					{		\draw[->-,thin] (\x,0)--(\x+1,0);		}
					\foreach \x in {2,6,8}
					{		\draw[->-,thin] (\x,0)--(\x-1,0);		}
					
				\begin{scope}[line width = 1.5pt]
					\draw[->-] (5.5,1.5)--(6,0);
					\draw[->-] (5.5,1.5)--(1,0);
					\draw[->-] (6.5,1.5)--(6,0);
					\draw[->-] (6.5,1.5)--(7,0);
					\draw (1,0)--(5,0);
					\draw (7,0)--(11,0);
				\foreach \x/\y in {2/1,4/5,8/7,10/11}
					{		\draw[->-] (\x,0)--(\y,0);		}	
				\end{scope}
	\end{tikzpicture}
	\caption{if $xv_{2i}, x'v_{2i}, xv_{2i+1} \in E(G)$}
		\label{fig:pro:2:2}
	\end{subfigure}
	\begin{subfigure}{0.475\textwidth}
		\centering
		\begin{tikzpicture}[scale=0.7]

				\foreach \x/\y in { 5.5/x',6.5/x}
					{
					\filldraw[black] (\x,1.5) circle (3pt);
					\node at (\x,1.9) {${\y}$};
					}
			
				\foreach \x/\y in { 1/1,5/2i-1,6/2i,7/2i+1,11/2\ell+1}
					{
					\filldraw[black] (\x,0) circle (3pt);
					\node at (\x,-0.4) {$v_{\y}$};
					}
				\foreach \x in {6,10,4}
					{		\draw[->-,thin] (\x,0)--(\x+1,0);		}
					\foreach \x in {2,6,8}
					{		\draw[->-,thin] (\x,0)--(\x-1,0);		}
					
				\begin{scope}[line width = 1.5pt]
					\draw[->-] (5.5,1.5)--(6,0);
					\draw[->-] (5.5,1.5)--(5,0);
					\draw[->-] (6.5,1.5)--(1,0);
					\draw[->-] (6.5,1.5)--(7,0);
					\draw (1,0)--(5,0);
					\draw (7,0)--(11,0);
					\foreach \x/\y in {2/1,4/5,8/7,10/11}
					{		\draw[->-] (\x,0)--(\y,0);		}	
				\end{scope}
	\end{tikzpicture}
	\caption{if $x'v_{2i-1}, x'v_{2i}, xv_{2i+1} \in E(G)$}
		\label{fig:pro:2:3}
	\end{subfigure}
	\caption{antipaths considered in the proof of Lemma~\ref{lem:property}\ref{itm:property:b}.}
	\label{fig:pro:2}
\end{figure}

Suppose that $i$ is $V_0$-surplus, that is, $e( V_0 , \{ v_{2i}, v_{2i+1}\}) \ge |V_0|+1$. Then there exists $x_i \in V_0$ such that $x_iv_{2i}, x_iv_{2i+1} \in E(G)$.
By~\ref{itm:property:b}, we deduce that $v_{2i} \notin N^+(V_0 \setminus \{x_i\})$ and so $V_0 \subseteq N^-(v_{2i+1})$. In particular, $e(V_0,\{v_{2i},v_{2i+1}\})=|V_0|+1$.
Again by~\ref{itm:property:b}, we have $v_{2i-1} \notin N^+(V_0)$.
Recall that $V_0 \subseteq N^-(v_1)$, so $i \ge 2$.
For all $x \in V_0$, note that $v_{2i}\dots v_1xv_{2i+1}\dots v_{2\ell+1}$ is an antipath of length $2\ell+1$.
Since there is no anticycle or antipath of length $2\ell +2$, we have $N^+(v_{2i})\subseteq V'$. Hence \ref{itm:property:c} holds.

By~\ref{itm:property:c}, $e(V_0,\{v_{2i},v_{2i+1}\})=|V_0|+1$ and $v_{2i-1}\notin N^+(V_0)$.
Therefore $e(V_0,\{v_{2i-2},v_{2i-1}\})=e(V_0,v_{2i-2})\le |V_0|$.
If $e(V_0,\{v_{2i-2},v_{2i-1},v_{2i},v_{2i+1}\})\ge 2|V_0|+1$,
then $e(V_0,v_{2i-2})=|V_0|$. Thus~\ref{itm:property:d} holds.

If $i$ and $i-1$ are $V_0$-surplus, then
\begin{align*}
	e(V_0,\{v_{2i-2},v_{2i-1},v_{2i},v_{2i+1}\}) = e(V_0,\{v_{2i-2},v_{2i-1}\}) + e(V_0,\{v_{2i},v_{2i+1}\}) \ge 2 (|V_0|+1).
\end{align*}
This contradicts~\ref{itm:property:d}.
Hence, \ref{itm:property:e} follows.



Suppose that $i$ is $V_0$-surplus and there exists $j\in[\ell-1]$ with $v_{2j+1}\in N^+(V_0)$ and $v_{2j}v_{2i}\in E(G)$.
By~\ref{itm:property:c}, there exist distinct $x_i , x \in V_0$ such that $x_{i} v_{2i}, x v_{2i+1}, x_{i} v_{2i+1} \in E(G)$ and $v_{2j+1} \in N^+(x_i) \cup N^+(x)$.
However, $G$ contains an antipath of length~$2\ell+2$, namely
\begin{align*}
	& v_{2i-1} \dots v_1 x_i v_{2i} v_{2j} \dots v_{2i+1} x v_{2j+1} \dots v_{2 \ell +1}
		& & \text{if $i <j$ and  $x v_{2j+1} \in E(G)$,}\\
	& v_{2i-1} \dots v_1 x v_{2i+1} \dots v_{2j} v_{2i} x_i v_{2j+1} \dots v_{2 \ell +1}
		& & \text{if $i <j$ and  $x_{i} v_{2j+1} \in E(G)$,}\\
	& v_{2i-1} \dots v_{2j+1} x v_{1} \dots v_{2j}  v_{2i} x_i v_{2i+1} \dots v_{2 \ell +1}
		& & \text{if $i  > j$ and  $x v_{2j+1} \in E(G)$,}\\
	& v_{2i-1} \dots v_{2j+1} x_i v_{2i} v_{2j} \dots v_{1} x v_{2i+1} \dots v_{2 \ell +1}
		& & \text{if $i  > j$ and  $x_i v_{2j+1} \in E(G)$,}
\end{align*}
(see Figure~\ref{fig:pro:8}), a contradiction.
Hence~\ref{itm:property:g} holds.
\begin{figure}
	\centering
	\begin{subfigure}{0.475\textwidth}
	\centering
		\begin{tikzpicture}[scale=0.7]

				\foreach \x/\y in { 6/x_i,11/x}
					{
					\filldraw[black] (\x,1.5) circle (3pt);
					\node at (\x,1.9) {${\y}$};
					}
			\filldraw[black] (5,0) circle (3pt);
				\foreach \x/\y in { 2.5/1,5/2i-1,6/2i,7/2i+1,10/2j,11/2j+1,13.5/2\ell+1}
					{
					\filldraw[black] (\x,0) circle (3pt);
					\node at (\x,-0.4) {$v_{\y}$};
					}

				\foreach \x in {6,10}
					{		\draw[->-,thin] (\x,0)--(\x+1,0);		}
					\foreach \x in {6,8}
					{		\draw[->-,thin] (\x,0)--(\x-1,0);		}
					
				\begin{scope}[line width =1.5pt]
					\draw[->-] (6,1.5)--(2.5,0);
					\draw[->-] (6,1.5)--(6,0);
					\draw[->-] (11,1.5)--(7,0);
					\draw[->-] (11,1.5)--(11,0);
					\draw (2.5,0)--(5,0);
					\draw (7,0)--(10,0);
					\draw (11,0)--(13.5,0);
					\draw[bend right, ->-] (10,0) to (6,0);
					\foreach \x/\y in {3.5/2.5,4/5,8/7,10/9,12/11,12.5/13.5}
					{		\draw[->-] (\x,0)--(\y,0);		}	
				\end{scope}
	\end{tikzpicture}
	\caption{if $i<j$  and $x v_{2j+1},v_{2j} v_{2i} \in E(G)$.}
		\label{fig:pro:8:1}
	\end{subfigure}
	\begin{subfigure}{0.475\textwidth}
	\centering
		\begin{tikzpicture}[scale=0.7]

				\foreach \x/\y in { 8.5/x_i,6/x}
					{
					\filldraw[black] (\x,1.5) circle (3pt);
					\node at (\x,1.9) {${\y}$};
					}
			\filldraw[black] (5,0) circle (3pt);
				\foreach \x/\y in { 2.5/1,5/2i-1,6/2i,7/2i+1,10/2j,11/2j+1,13.5/2\ell+1}
					{
					\filldraw[black] (\x,0) circle (3pt);
					\node at (\x,-0.4) {$v_{\y}$};
					}
				\foreach \x in {6,10}
					{		\draw[->-,thin] (\x,0)--(\x+1,0);		}
					\foreach \x in {6,8}
					{		\draw[->-,thin] (\x,0)--(\x-1,0);		}
					
				\begin{scope}[line width =1.5pt]
					\draw[->-] (8.5,1.5)--(11,0);
					\draw[->-] (8.5,1.5)--(6,0);
					\draw[->-] (6,1.5)--(2.5,0);
					\draw[->-] (6,1.5)--(7,0);
					\draw (2.5,0)--(5,0);
					\draw (7,0)--(10,0);
					\draw (11,0)--(13.5,0);
					\draw[bend right, ->-] (10,0) to (6,0);
					\foreach \x/\y in {3.5/2.5,4/5,8/7,10/9,12/11,12.5/13.5}
					{		\draw[->-] (\x,0)--(\y,0);		}	
				\end{scope}
	\end{tikzpicture}
	\caption{if $i<j$  and $x_i v_{2j+1},v_{2j} v_{2i} \in E(G)$.}
		\label{fig:pro:8:2}
	\end{subfigure}\\
	\begin{subfigure}{0.475\textwidth}
	\centering
		\begin{tikzpicture}[scale=0.7]

				\foreach \x/\y in { 10/x_i,6/x}
					{
					\filldraw[black] (\x,1.5) circle (3pt);
					\node at (\x,1.9) {${\y}$};
					}
	
				\foreach \x/\y in { 2.5/1,9/2i-1,10/2i,11/2i+1,6/2j,7/2j+1,13.5/2\ell+1}
					{
					\filldraw[black] (\x,0) circle (3pt);
					\node at (\x,-0.4) {$v_{\y}$};
					}
				\foreach \x in {6,10}
					{		\draw[->-,thin] (\x,0)--(\x+1,0);		}
					\foreach \x in {6,8,10}
					{		\draw[->-,thin] (\x,0)--(\x-1,0);		}
					
				\begin{scope}[line width =1.5pt]
					\draw[->-] (6,1.5)--(2.5,0);
					\draw[->-] (6,1.5)--(7,0);
					\draw[->-] (10,1.5)--(10,0);
					\draw[->-] (10,1.5)--(11,0);
					\draw (2.5,0)--(6,0);
					\draw (7,0)--(9,0);
					\draw (11,0)--(13.5,0);
					\draw[bend left, ->-] (6,0) to (10,0);
					\foreach \x/\y in {3.5/2.5,6/5,8/7,12/11,12.5/13.5}
					{		\draw[->-] (\x,0)--(\y,0);		}	
				\end{scope}
	\end{tikzpicture}
	\caption{if $i>j$  and $x v_{2j+1},v_{2j} v_{2i} \in E(G)$.}
		\label{fig:pro:8:3}
	\end{subfigure}
	\begin{subfigure}{0.475\textwidth}
	\centering
		\begin{tikzpicture}[scale=0.7]

				\foreach \x/\y in { 8.5/x_i,6/x}
					{
					\filldraw[black] (\x,1.5) circle (3pt);
				\node at (\x,1.9) {${\y}$};
					}

				\foreach \x/\y in { 2.5/1,9/2i-1,10/2i,11/2i+1,6/2j,7/2j+1,13.5/2\ell+1}
					{
					\filldraw[black] (\x,0) circle (3pt);
				\node at (\x,-0.4) {$v_{\y}$};
					}
				\foreach \x in {6,10}
					{		\draw[->-,thin] (\x,0)--(\x+1,0);		}
					\foreach \x in {6,8,10}
					{		\draw[->-,thin] (\x,0)--(\x-1,0);		}
					
				\begin{scope}[line width =1.5pt]
					\draw[->-] (6,1.5)--(2.5,0);
					\draw[->-] (6,1.5)--(11,0);
					\draw[->-] (8.5,1.5)--(10,0);
					\draw[->-] (8.5,1.5)--(7,0);
					\draw (2.5,0)--(6,0);
					\draw (7,0)--(9,0);
					\draw (11,0)--(13.5,0);
					\draw[bend left=25, ->-] (6,0) to (10,0);
					\foreach \x/\y in {3.5/2.5,6/5,8/7,12/11,12.5/13.5}
					{		\draw[->-] (\x,0)--(\y,0);		}	
				\end{scope}
	\end{tikzpicture}
	\caption{if $i>j$  and $x_i v_{2j+1},v_{2j} v_{2i} \in E(G)$.}
		\label{fig:pro:8:4}
	\end{subfigure}
	\caption{antipaths considered in the proof of Lemma~\ref{lem:property}\ref{itm:property:g}.}
	\label{fig:pro:8}
\end{figure}


If $i$ and $i'$ are $V_0$-surplus, then $v_{2i+1},v_{2i'+1}\in N^+(V_0)$ by~\ref{itm:property:c}.
By~\ref{itm:property:g} with $\{i,j\}= \{i,i'\}$, we have $v_{2i}v_{2i'}\notin E(G)$, implying~\ref{itm:property:h}.

Let $I = \{ i \in [\ell-1] \colon \text{ $i$ is $V_0$-surplus}\}$ and $V_{2I} = \{v_{2i} \colon i \in I\}$.
Consider $i \in I$.
By~\ref{itm:property:h},
\begin{align}
	\left( N^+(v_{2i}) \cup N^-(v_{2i}) \right)  \cap V_{2I}  = \emptyset.
	\label{eqn::property:9}
\end{align}
By~\ref{itm:property:c}, we have $V_0 \subseteq N^-(v_{2i+1})$, $d^-(v_{2i},V_0) =1$, and $N^+(v_{2i})\subseteq V'$.
By~\eqref{eqn::property:9}, we have
\begin{align*}
	d^+(v_{2i},V' \setminus V_{2I}) =  d^+(v_{2i}, V')  = d^+(v_{2i}) \ge k.
\end{align*}
Let $x_i\in V_0$ be such that $x_iv_{2i}\in E(G)$. Thus $d^-(v_{2i}) \ge 1$ and so $d^-(v_{2i})\ge k$.
Note that
\begin{align*}
		|N^-(v_{2i}) \setminus (V_0 \cup V')|
		& = d^-(v_{2i}) - d^-(v_{2i}, V_0)  - d^-(v_{2i},V'\setminus V_{2I}) - d^-(v_{2i}, V_{2I}) \\
		& \ge k - 1 -\left( |V' \setminus V_{2I}| - d^+(v_{2i},V' \setminus V_{2I})  \right)-0
		\\
		& \ge	k-1-(2\ell - |I| -  k )  = |I|+ 2(k-\ell) -1 \ge |I| + 3=m+3.
\end{align*}
Hence~\ref{itm:property:i} holds.

Let $i$ be $V_0$-surplus.
By~\ref{itm:property:c}, $V_0 \subseteq N^-(v_{2i+1})$.
If $v_{2i}v_{2\ell+1}\in E(G)$, then $v_1 \dots v_{2i} v_{2\ell+1} \dots v_{2i+1} x v_1$ is an anticycle of length $2\ell +2$ for all $x\in V_0$ (see Figure~\ref{fig:pro:5:1}), a contradiction. Hence $v_{2i}v_{2\ell+1}\notin E(G)$.
By~\ref{itm:property:c}, we have that $N^+(v_{2i})\subseteq V'$ and that there exists exactly one $x_i\in V_0$ with $x_iv_{2i}\in E(G)$.
By~\ref{itm:property:i}, $|N^{-}(v_{2i}) \setminus (V_0 \cup V')| \ge 4$.
Let $x \in V_0 \setminus \{x_i\}$ and $z \in N^-(v_{2i}) \setminus  (V_0 \cup V' \cup \{v_{2\ell+1}\}) $.
If $v_{2i-2} v_{2\ell+1} \in E(G)$, then $G$ contains an antipath of length $2\ell +2$, namely $xv_1 \dots v_{2i-2} v_{2\ell+1} \dots v_{2i+1} x_i v_{2i} z $ (see Figure~\ref{fig:pro:5:2}), a contradiction.
Hence~\ref{itm:property:f} holds.
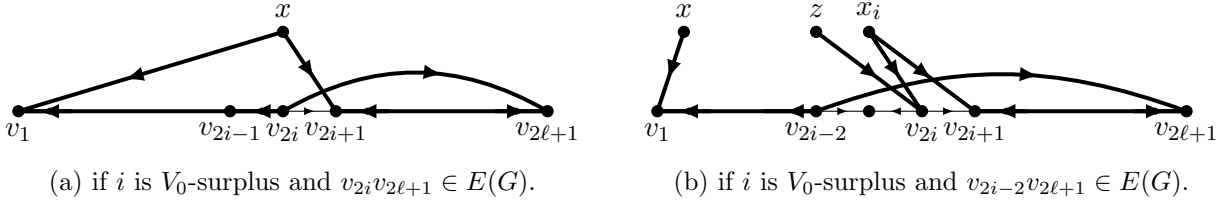
\begin{figure}
	\centering
		\begin{subfigure}{0.475\textwidth}
			\centering
		\begin{tikzpicture}[scale=0.7]

				\foreach \x/\y in { 6/x}
					{
					\filldraw[black] (\x,1.5) circle (3pt);
					\node at (\x,1.9) {${\y}$};
					}
			
				\foreach \x/\y in { 1/1,5/2i-1,6/2i,7/2i+1,11/2\ell+1}
					{
					\filldraw[black] (\x,0) circle (3pt);
					\node at (\x,-0.4) {$v_{\y}$};
					}
				\foreach \x in {6,10}
					{		\draw[->-,thin] (\x,0)--(\x+1,0);		}
					\foreach \x in {2,6,8}
					{		\draw[->-,thin] (\x,0)--(\x-1,0);		}
					
				\begin{scope}[line width =1.5pt]
					\draw[->-] (6,1.5)--(7,0);
					\draw[->-] (6,1.5)--(1,0);
					\draw (1,0)--(6,0);
					\draw (7,0)--(11,0);
					\draw[bend left, ->-] (6,0) to (11,0);
					\foreach \x/\y in {2/1,6/5,8/7,10/11}
					{		\draw[->-] (\x,0)--(\y,0);		}	
				\end{scope}
	\end{tikzpicture}
	\caption{if $i$ is $V_0$-surplus and $v_{2i}v_{2\ell+1} \in E(G)$.}
		\label{fig:pro:5:1}
	\end{subfigure}
	\begin{subfigure}{0.475\textwidth}
		\centering
		\begin{tikzpicture}[scale=0.7]

				\foreach \x/\y in { 1.5/x,4/z,5/x_i}
					{
					\filldraw[black] (\x,1.5) circle (3pt);
					\node at (\x,1.9) {${\y}$};
					}
			\filldraw[black] (5,0) circle (3pt);
				\foreach \x/\y in { 1/1,4/2i-2,6/2i,7/2i+1,11/2\ell+1}
					{
					\filldraw[black] (\x,0) circle (3pt);
					\node at (\x,-0.4) {$v_{\y}$};
					}
				\foreach \x in {6,10,4}
					{		\draw[->-,thin] (\x,0)--(\x+1,0);		}
					\foreach \x in {2,6,8}
					{		\draw[->-,thin] (\x,0)--(\x-1,0);		}
					
				\begin{scope}[line width =1.5pt]
					\draw[->-] (1.5,1.5)--(1,0);
					\draw[->-] (4,1.5)--(6,0);
					\draw[->-] (5,1.5)--(6,0);
					\draw[->-] (5,1.5)--(7,0);
					\draw (1,0)--(4,0);
					\draw (7,0)--(11,0);
					\draw[bend left=20, ->-] (4,0) to (11,0);
					\foreach \x/\y in {2/1,4/3,8/7,10/11}
					{		\draw[->-] (\x,0)--(\y,0);		}	
				\end{scope}
	\end{tikzpicture}
	\caption{if $i$ is $V_0$-surplus and $v_{2i-2}v_{2\ell+1} \in E(G)$.}
		\label{fig:pro:5:2}
	\end{subfigure}
		\caption{anticycle and antipath considered in the proof of Lemma~\ref{lem:property}\ref{itm:property:f}.}
	\label{fig:pro:5}
\end{figure}

Suppose that $i$ and $i'$ are distinct $V_0$-surplus integers and $j\in[\ell-1]$ such that  $v_{2j+1}\in N^+(V_0)$ and $v_{2j} \in N^+(v_{2i}) \cap N^+(v_{2i'})$.
Without loss of generality, assume that $i<i'$.
By~\ref{itm:property:c} and $|V_0|\ge 2$, there exist distinct $x,x' \in V_0$ such that $xv_{2i+1}, xv_{2i'+1}, x'v_{2i+1}, x'v_{2i'+1} \in E(G)$.
Then $G$ contains an antipath of length~$2\ell+2$, namely
\begin{align*}
	&v_{2j-1} \dots v_{2i'+1} x' v_1 \dots v_{2i}  v_{2j} v_{2i'} \dots v_{2i+1}	x v_{2j+1} \dots v_{2\ell+1}
	& &\text{if $i < i' <j$,}\\
	& v_{2j-1} \dots v_{2i+1} x' v_1 \dots v_{2i}  v_{2j} v_{2i'} \dots v_{2j+1}	x v_{2i'+1} \dots v_{2\ell+1}
		& & \text{if $i < j<i' $,}\\
	& v_{2j-1} \dots v_{1} x v_{2j+1} \dots v_{2i}  v_{2j} v_{2i'} \dots v_{2i+1}	x' v_{2i'+1} \dots v_{2\ell+1}
		& & \text{if $j < i<i' $,}
\end{align*}
(see Figure~\ref{fig:pro:7}), a contradiction.
Hence~\ref{itm:property:j} holds.
\begin{figure}
	\centering
	\begin{subfigure}{1\textwidth}
	\centering
		\begin{tikzpicture}[scale=0.8]

				\foreach \x/\y in {6/x',13/x}
					{
					\filldraw[black] (\x,1.5) circle (3pt);
					\node at (\x,1.9) {${\y}$};
					}
				\foreach \x/\y in {4/1,6/2i,7/2i+1,10/2i',11/2i'+1,13/2j-1,14/2j,15/2j+1,18/2\ell+1}
					{
					\filldraw[black] (\x,0) circle (3pt);
					\node at (\x,-0.4) {$v_{\y}$};
					}
				\foreach \x/\y in {6/7,10/11,14/13,14/15}
					{		\draw[->-,thin] (\x,0)--(\y,0);		}
					
				\begin{scope}[line width =1.5pt]
					\draw[->-] (6,1.5)--(4,0);
					\draw[->-] (6,1.5)--(11,0);
					\draw[->-] (13,1.5)--(7,0);
					\draw[->-] (13,1.5)--(15,0);
					\draw (4,0)--(6,0);
					\draw (10,0)--(7,0);
					\draw (15,0)--(18,0);
					\draw (11,0)--(13,0);
					\draw[bend left, ->-] (10,0) to (14,0);
					\draw[bend left, ->-] (6,0) to (14,0);
					\foreach \x/\y in {4.5/4,6/5,8/7,10/9,12/11,16/15,17/18,12/13}
					{		\draw[->-] (\x,0)--(\y,0);		}	
				\end{scope}
	\end{tikzpicture}
	\caption{if $i < i' <j$.}
		\label{fig:pro:7:1}
	\end{subfigure}
	\begin{subfigure}{1\textwidth}
	\centering
		\begin{tikzpicture}[scale=0.8]

				\foreach \x/\y in { 8/x',12/x}
					{
					\filldraw[black] (\x,1.5) circle (3pt);
					\node at (\x,1.9) {${\y}$};
					}
				\foreach \x/\y in {4/1,6/2i,7/2i+1,14/2i',15/2i'+1,9/2j-1,10/2j,11/2j+1,18/2\ell+1}
					{
					\filldraw[black] (\x,0) circle (3pt);
					\node at (\x,-0.4) {$v_{\y}$};
					}
				\foreach \x/\y in {6/7,10/11,10/9,14/15}
					{		\draw[->-,thin] (\x,0)--(\y,0);		}
					
				\begin{scope}[line width =1.5pt]
					\draw[->-] (8,1.5)--(4,0);
					\draw[->-] (8,1.5)--(7,0);
					\draw[->-] (12,1.5)--(11,0);
					\draw[->-] (12,1.5)--(15,0);
					\draw (4,0)--(6,0);
					\draw (7,0)--(9,0);
					\draw (15,0)--(18,0);
					\draw (11,0)--(14,0);
					\draw[bend right, ->-] (14,0) to (10,0);
					\draw[bend left, ->-] (6,0) to (10,0);
					\foreach \x/\y in {4.5/4,6/5,8/7,8/9,12/11,14/13,16/15,17/18}
					{		\draw[->-] (\x,0)--(\y,0);		}	
				\end{scope}
	\end{tikzpicture}
	\caption{if $i < j<i' $.}
		\label{fig:pro:7:2}
	\end{subfigure}
	\begin{subfigure}{1\textwidth}
	\centering
		\begin{tikzpicture}[scale=0.8]

				\foreach \x/\y in {6/x,16/x'}
					{
					\filldraw[black] (\x,1.5) circle (3pt);
					\node at (\x,1.9) {${\y}$};
					}
				\foreach \x/\y in {3/1,10/2i,11/2i+1,14/2i',15/2i'+1,5/2j-1,6/2j,7/2j+1,17.5/2\ell+1}
					{
					\filldraw[black] (\x,0) circle (3pt);
					\node at (\x,-0.4) {$v_{\y}$};
					}
				\foreach \x/\y in {6/5,6/7,10/11,14/15}
					{		\draw[->-,thin] (\x,0)--(\y,0);		}
					
				\begin{scope}[line width =1.5pt]
					\draw[->-] (6,1.5)--(3,0);
					\draw[->-] (6,1.5)--(7,0);
					\draw[->-] (16,1.5)--(11,0);
					\draw[->-] (16,1.5)--(15,0);
					\draw (3,0)--(5,0);
					\draw (7,0)--(10,0);
					\draw (15,0)--(17.5,0);
					\draw (11,0)--(14,0);
					\draw[bend right, ->-] (14,0) to (6,0);
					\draw[bend right, ->-] (10,0) to (6,0);
					\foreach \x/\y in {4/3,4/5,8/7,10/9,12/11,14/13,16/15,16.5/17.5}
					{		\draw[->-] (\x,0)--(\y,0);		}	
				\end{scope}
	\end{tikzpicture}
	\caption{if $j < i<i' $.}
		\label{fig:pro:7:3}
	\end{subfigure}

	\caption{$i$ and $i'$ are $V_0$-surplus and $v_{2j} \in N^+(v_{2i}) \cap N^+(v_{2i'})$.}
	\label{fig:pro:7}
\end{figure}
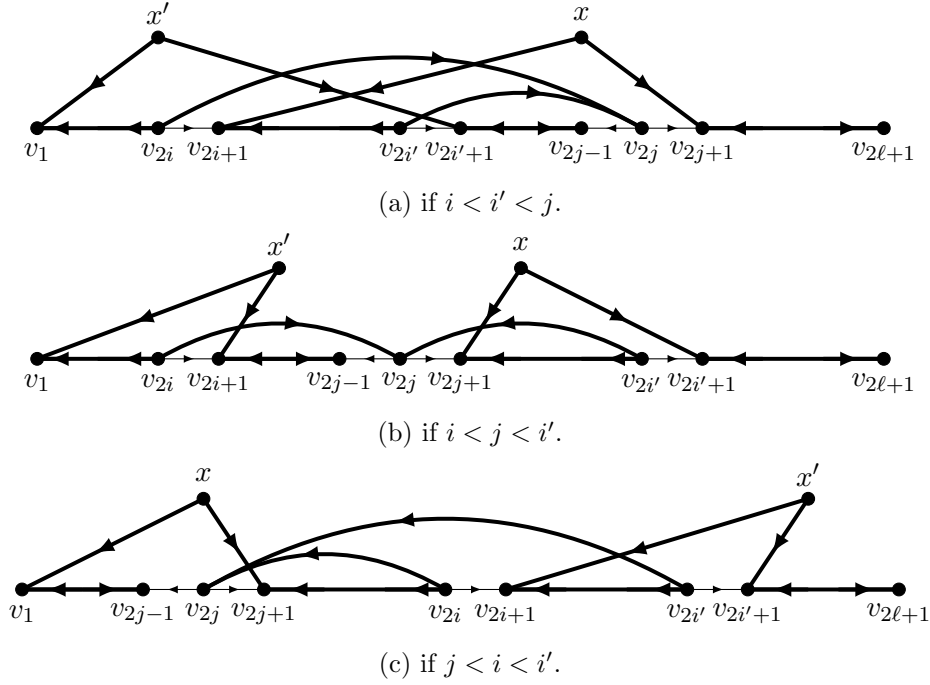
\end{proof}

The following lemma shows that, if the length of a longest antipath is fewer than~$2k-1$, then each of the two endvertices of the longest antipath can be replaced independently by at least one other vertex.
This structural result is essential to the proof of our main theorem. Since its proof is rather involved, we defer it to Section~\ref{sec:lemmaproof}.

\begin{lemma} \label{lem:twofatends}
Let $k \in \mathbb{N}$ and let $G$ be an oriented graph with $\pd(G) \ge k$.
Suppose that a longest antipath in~$G$ has length~$2\ell+1 < 2k-1$.
Then there exists an antipath blow-up $V_0v_1\dots v_{2\ell}V_{2\ell+1}$ of length $2 \ell +1$ with $|V_0|, |V_{2\ell+1}| \ge 2$.
\end{lemma}

We now prove Theorem~\ref{thm:pseudo} assuming Lemma~\ref{lem:twofatends}.

\begin{proof}[Proof of Theorem~\ref{thm:pseudo}]
If $k=1$, then the result is immediate.
Hence we may assume $k\ge2$.
Let $m$ be the length of a longest antipath in~$G$.
We may assume that $m < 2k-1$ or else we are done.
By Lemma~\ref{lem:oddlength}, we have $m = 2 \ell +1$ for some $ \ell < k-1$.
On the other hand $\pd(G) \ge k \ge 2$, so $m \ge 2$ and moreover $\ell \ge 1$.
By Lemma~\ref{lem:twofatends}, there exists an antipath blow-up $V_0v_1\dots v_{2\ell}V_{2\ell+1}$ of length $2 \ell +1$ with $|V_0|, |V_{2\ell+1}| \ge 2$.

For $i \in [\ell-1]$, set
\begin{align*}
	\phi(i) = e(V_0,\{v_{2i} , v_{2i+1}\})+e(\{v_{2i}, v_{2i+1}\},V_{2\ell+1}).
\end{align*}
We say that $i$ is \emph{surplus} if $\phi(i) \ge |V_0|+|V_{2\ell+1}|+1$.
If $i$ is surplus, then $i$ is $V_0$-surplus or $V_{2\ell+1}$-surplus.
Let $V' = \{v_j \colon j \in [2\ell] \}$.

In the next claim, we study some properties of being surplus.

\begin{claim} \label{clm:surplus}
Suppose that $i \in [\ell-1]$ is surplus.
Then the following statements hold:
\begin{enumerate}[label={\rm (\alph*)}]
	\item \label{itm:surplus:0}
	$i$ is either $V_0$-surplus or $V_{2\ell+1}$-surplus;
	\item \label{itm:surplus:1}
	if $i$ is $V_0$-surplus, then $\phi(i-1)+\phi(i) \le 2(|V_0|+|V_{2\ell+1}|)$;
	\item \label{itm:surplus:2}
	if $i$ is $V_{2\ell+1}$-surplus, then $\phi(i) +\phi(i+1) \le 2(|V_0|+|V_{2\ell+1}|)$;
	\item \label{itm:surplus:3}
	if $i$ is $V_{0}$-surplus and $i-2$ is both surplus and $V_{2\ell+1}$-surplus, then
	$\phi(i)+\phi(i-1)+\phi(i-2) \le 3(|V_0|+|V_{2\ell+1}|)$.
\end{enumerate}
\end{claim}

\begin{proofclaim}
Suppose that $i$ is both $V_0$-surplus and surplus.
By Lemma~\ref{lem:property}\ref{itm:property:c} and~\ref{itm:property:f}, we have $e(V_0,\{v_{2i},v_{2i+1}\})=|V_0|+1$ and $d^+(v_{2i},V_{2\ell+1})=0$, respectively.
Hence we have
\begin{align*}
|V_0| +|V_{2\ell+1}|+1 & \le \phi(i) = e(V_0,\{v_{2i},v_{2i+1}\}) +d^+(v_{2i},V_{2\ell+1}) +d^+(v_{2i+1},V_{2\ell+1}) \\
& = |V_0| +1+0+d^+(v_{2i+1},V_{2\ell+1})\le |V_0|+1 +|V_{2\ell+1}|.
\end{align*}
It follows that $d^+(v_{2i+1},V_{2\ell+1})=|V_{2\ell+1}|$ and
\begin{align}
	\phi(i) = |V_0|+|V_{2\ell+1}|+1. \label{eqn:phi(i)}
\end{align}
In particular, $e(\{v_{2i},v_{2i+1}\},V_{2\ell+1})=|V_{2\ell+1}|$,
so $i$ is not $V_{2\ell+1}$-surplus, implying~\ref{itm:surplus:0}.

Again, by Lemma~\ref{lem:property}\ref{itm:property:c} and~\ref{itm:property:f}, we have $v_{2i-1} \notin N^+(V_{0})$ and $v_{2i-2} \notin N^-(V_{2\ell+1})$, respectively.
Therefore $\phi(i-1)=e(V_0,v_{2i-2})+e(v_{2i-1},V_{2\ell+1})$.
If $\phi(i-1) \ge  |V_0|+|V_{2\ell+1}|$, then $e(V_0,v_{2i-2})=|V_0|$ and $e(v_{2i-1},V_{2\ell+1})=|V_{2\ell+1}|$. Since $|V_0|, |V_{2\ell+1}| \ge 2$, there exist distinct $x,x' \in V_0 \cap N^{-}(v_{2i-2})$ and $y,y' \in V_{2\ell+1} \cap N^{+}(v_{2i-1})$.
However, $v_{2i-3} \dots v_1 x v_{2i-2} x' v_{2i+1} \dots v_{2\ell} y v_{2i-1} y'$ is an antipath of length $2\ell +2$ (see Figure~\ref{fig:surplus:1}), a contradiction.
Thus  $\phi(i-1)\le  |V_0|+|V_{2\ell+1}|-1$.
Together with~\eqref{eqn:phi(i)}, \ref{itm:surplus:1} holds.
Note that \ref{itm:surplus:2} is proved analogously by considering the reverse of~$G$.

We now prove~\ref{itm:surplus:3}.
Suppose that both $i$ and $i-2$ are surplus, $i$ is $V_{0}$-surplus, and $i-2$ is $V_{2\ell+1}$-surplus.
By Lemma~\ref{lem:property}\ref{itm:property:c} applied to the reverse of $G$, there exists a unique $y_{i-2}\in V_{2\ell+1}$ such that $v_{2i-3}y_{i-2}, v_{2i-4}y_{i-2}\in E(G)$.
We claim that $e(V_0,\{v_{2i-2},v_{2i-1}\})\le 1$.
Suppose not and recall that $V_0\cap N^-(v_{2i-1})=\emptyset$ by Lemma~\ref{lem:property}\ref{itm:property:c}.
Then there exist distinct $x,x'\in V_0\cap N^-(v_{2i-2})$.
Since $|V_{2\ell+1}|\ge2$, there exists $y\in V_{2\ell+1}\setminus\{y_{i-2}\}$.
Since $|N^+(v_{2i-3})\setminus(V_{2\ell+1}\cup V')|\ge 4$ by Lemma~\ref{lem:property}\ref{itm:property:i} (applied to reverse of~$G$), there exists $z\in N^+(v_{2i-3})\setminus (V_{2\ell+1}\cup V'\cup\{x,x'\})$.
Since $i$ is $V_0$-surplus, Lemma~\ref{lem:property}\ref{itm:property:c} implies that $x'v_{2i+1}\in E(G)$.
Then $zv_{2i-3}y_{i-2}v_{2i-4}\dots v_1xv_{2i-2}x'v_{2i+1}\dots v_{2\ell}y$ is an antipath of length $2\ell+2$ (see Figure~\ref{fig:surplus:2}), a contradiction.
Therefore $e(V_0,\{v_{2i-2},v_{2i-1}\})\le1$.
By a similar argument (for the reverse of~$G$), we deduce that $e(\{v_{2i-2},v_{2i-1}\},V_{2\ell+1})\le1$.
Hence $\phi(i-1) \le 2 \le |V_0|+|V_{2\ell+1}|-2$.
Together with~\eqref{eqn:phi(i)} applied to the indices $i$ and $i-2$, \ref{itm:surplus:3} follows.
\begin{figure}
			\centering
	\begin{subfigure}{\textwidth}
		\centering
		\begin{tikzpicture}[scale=0.8]

				\foreach \x/\y in { 5/x,7/x'}
					{
					\filldraw[black] (\x,1.5) circle (3pt);
					\node at (\x,1.9) {${\y}$};
					}
				\foreach \x/\y in {6.5/y',8/y}
					{
					\filldraw[black] (\x,-1.5) circle (3pt);
					\node at (\x,-1.9) {${\y}$};
					}
			
				\foreach \x/\y in { 1/1,5/2i-3,6/2i-2,8/2i,9/2i+1,13/2\ell}
					{
					\filldraw[black] (\x,0) circle (3pt);
					\node at (\x,-0.4) {$v_{\y}$};
					}
					\filldraw[black] (7,0) circle (3pt);	
					\node at (7,0.4) {$v_{2i-1}$};

				\foreach \x/\y in {6/5,6/7,8/9,8/7}
					{		\draw[->-,thin] (\x,0)--(\y,0);		}
					
				\begin{scope}[line width = 1.5pt]
					\draw[->-] (5,1.5)--(1,0);
					\draw[->-] (5,1.5)--(6,0);
					\draw[->-] (7,1.5)--(6,0);
					\draw[->-] (7,1.5)--(9,0);
					\draw[->-] (13,0)--(8,-1.5);
					\draw[->-] (7,0)--(8,-1.5);
					\draw[->-] (7,0)--(6.5,-1.5);
					\draw (1,0)--(5,0);
					\draw (9,0)--(13,0);
					\foreach \x/\y in {2/1,4/5,10/9,13/12}
					{		\draw[->-] (\x,0)--(\y,0);		}	
				\end{scope}
	\end{tikzpicture}
	\caption{if $\phi(i-1)  + \phi(i)> 2  |V_0| + 2 |V_{2\ell+1}|$.}
	\label{fig:surplus:1}
	\end{subfigure}
	\begin{subfigure}{\textwidth}
		\centering
		\begin{tikzpicture}[scale=0.8]

				\foreach \x/\y in { 5/x,9/x'}
					{
					\filldraw[black] (\x,1.5) circle (3pt);
					\node at (\x,1.9) {${\y}$};
					}
				\foreach \x/\y in { 11/y,4/y_{i-2},6/z}
					{
					\filldraw[black] (\x,-1.5) circle (3pt);
					\node at (\x,-1.9) {${\y}$};
					}
			
				\foreach \x/\y in { 1/1,6/2i-2,7/2i-1,8/2i,9/2i+1,13/2\ell}
					{
					\filldraw[black] (\x,0) circle (3pt);
					\node at (\x,-0.4) {$v_{\y}$};
					}
					\filldraw[black] (4,0) circle (3pt);	
					\node at (4,0.4) {$v_{2i-4}$};
					\filldraw[black] (5,0) circle (3pt);	
					\node at (5,0.4) {$v_{2i-3}$};

				\foreach \x/\y in {6/5,6/7,8/9,8/7,4/5}
					{		\draw[->-,thin] (\x,0)--(\y,0);		}
					
				\begin{scope}[line width = 1.5pt]
					\draw[->-] (5,1.5)--(1,0);
					\draw[->-] (5,1.5)--(6,0);
					\draw[->-] (9,1.5)--(6,0);
					\draw[->-] (9,1.5)--(9,0);
					\draw[->-] (13,0)--(11,-1.5);
					\draw[->-] (4,0)--(4,-1.5);
					\draw[->-] (5,0)--(4,-1.5);
					\draw[->-] (5,0)--(6,-1.5);
					\draw (1,0)--(4,0);
					\draw (9,0)--(13,0);
					\foreach \x/\y in {2/1,4/3,10/9,13/12}
					{		\draw[->-] (\x,0)--(\y,0);		}	
				\end{scope}
	\end{tikzpicture}
	\caption{if $\phi(i-2)+\phi(i-1)+\phi(i) >  3|V_0|+3|V_{2\ell+1}|$.}
	\label{fig:surplus:2}
	\end{subfigure}
	\caption{antipaths considered in the proof of Claim~\ref{clm:surplus}.}
	\label{fig:surplus}
\end{figure}
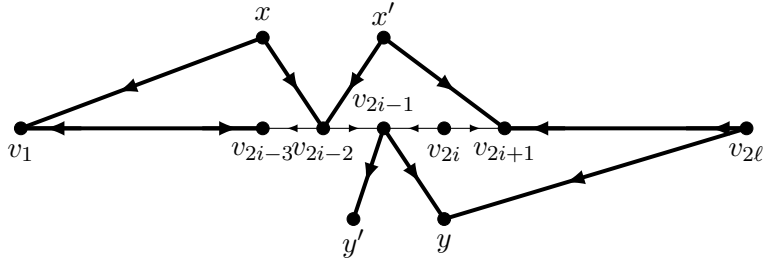
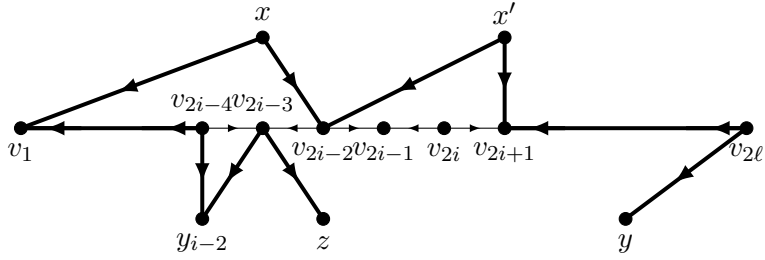
\end{proofclaim}

Let $I$ be the set of $i \in [\ell]$ that are surplus and $V_0$-surplus and let $I^-=\{i-1 \colon i \in I\}$.
Similarly, let $J$ be the set of $j \in [\ell]$ that are surplus and $V_{2\ell+1}$-surplus and let $J^+ =\{j+1 \colon j \in J\}$.
By Lemma~\ref{lem:property}\ref{itm:property:c} to both~$G$ and the reverse of~$G$, we have $I \subseteq [2,\ell-1]$ and $J \subseteq [1,\ell-2]$.
By Claim~\ref{clm:surplus}, we have
\begin{align}
	I \cap J = \emptyset	\qquad \text{and} \qquad (I^- \cup J^+) \cap (I \cup J) = \emptyset.
	\label{eqn:AB}
\end{align}
For all $i \in I^- \cup  J^+$, set
\begin{align*}
	S_i =
	\begin{cases}
		\{i,i+1\} & \text{if $i \in I^- \setminus  J^+$},\\
		\{i-1,i\} & \text{if $i \in J^+ \setminus  I^-$},\\
		\{i-1, i,i+1\} & \text{if $i \in I^- \cap  J^+$}.
	\end{cases}
\end{align*}
By~\eqref{eqn:AB}, $\{S_i \colon i \in I^- \cup  J^+ \}$ is a set of disjoint intervals in~$[\ell-1]$.
Let $S = \bigcup\{S_i \colon i \in I^- \cup  J^+ \}$, so $I \cup J \subseteq S$.
Moreover, Claim~\ref{clm:surplus} implies that $\sum_{j\in S_i}\phi(j)\le |S_i|( |V_0| + |V_{2\ell+1}| )$ for all $i\in I^-\cup J^+$. Thus
\begin{align*}
	\sum_{j \in S} \phi(j) =
	\sum_{i \in I^- \cup  J^+ } \sum_{j \in S_i} \phi(j)
	\le \sum_{i \in I^- \cup  J^+ } |S_i|( |V_0| + |V_{2\ell+1}| )
	= |S|( |V_0| + |V_{2\ell+1}| ).
\end{align*}
Since $I\cup J\subseteq S$, for all $j\in[\ell-1]\setminus S$, $j$ is not surplus and so $\phi(j)\le ( |V_0| + |V_{2\ell+1}| )$.
Therefore,
\begin{align*}
	e(V_0,V')+e(V',V_{2\ell+1})
	& =  e(V_0,\{v_1,v_{2\ell}\})+e(\{v_1,v_{2\ell}\},V_{2\ell+1}) + \sum_{j \in S } \phi (j) +\sum_{j \in [\ell-1] \setminus S }  \phi (j) \\
	& \le \left( 2 + |S| + | [\ell-1] \setminus S | \right)( |V_0| + |V_{2\ell+1}| )
	= (\ell +1)( |V_0| + |V_{2\ell+1}| ) .
\end{align*}
On the other hand, by Lemma~\ref{lem:property}\ref{itm:property:a} (applied to both $G$ and the reverse of $G$), we have $ N^+(V_0), N^-(V_{2\ell+1}) \subseteq V'$.
By the assumption $\pd(G)\ge k$, we have $d^+(x)\ge k$ and $d^-(y)\ge k$ for all $x\in V_0$
and $y\in V_{2\ell+1}$.
Hence
\begin{align*}
	k ( |V_0| + |V_{2\ell+1}| )
	& \le \sum_{x \in V_0} d^+(x) + \sum_{y \in V_{2\ell+1}} d^-(y)
	= e(V_0,V')+e(V',V_{2\ell+1})
	 \le (\ell +1)( |V_0| + |V_{2\ell+1}| ).
\end{align*}
This implies that $\ell \ge k-1$, a contradiction.
\end{proof}


\section{Proof of Lemma~\ref{lem:twofatends}} \label{sec:lemmaproof}

We first show that there exists an antipath blow-up $V_0v_1\dots v_{2\ell}V_{2\ell+1}$ such that at least one of~$V_0$ and~$V_{2\ell+1}$ has size at least~$2$.
This result has already been proved by  Grzesik and Skrzypczyk~\cite[Lemma~10]{GS2025}.
We include its proof as it helps to illustrate some ideas used in the proof of Lemma~\ref{lem:twofatends}.

\begin{lemma} \label{lem:onefatend}
Let $k \in \mathbb{N}$ and let $G$ be an oriented graph with $\pd(G) \ge k$.
Suppose that a longest antipath in~$G$ has length~$2\ell+1 < 2k-1$.
Then there exists an antipath blow-up $V_0v_1\dots v_{2\ell}V_{2\ell+1}$ of length $2 \ell +1$ with $\max\{ |V_0|, |V_{2\ell+1}|\} \ge 2$.
\end{lemma}

\begin{proof}
When $k=1$, the lemma is vacuously true, so we may assume that $k \ge2$.
Let $P = v_0v_1\dots v_{2\ell+1}$ be a longest antipath in~$G$, where the edges of $P$ are directed from even-indexed vertices to odd-indexed vertices.
Let $V^* = \{ v_j\colon j \in \{0\} \cup [2\ell] \}$. Since $2\ell+1<2k-1$, we have $2\ell+2<2k-1$.
Hence, by Lemma~\ref{lem:cycle}, $G$ does not contain an anticycle of length $2\ell+2$. We claim that
\begin{align*}
N^+(v_0)\subseteq V^*.
\end{align*}
Indeed, if $v_0$ has an out-neighbour $x\notin V(P)$, then $xv_0v_1\dots v_{2\ell+1}$
is an antipath of length $2\ell+2$, a contradiction. If $v_0v_{2\ell+1}\in E(G)$, then
$v_0v_1\dots v_{2\ell+1}v_0$ is an anticycle of length $2\ell+2$, again a contradiction.

\begin{claim} \label{clm:fatend}
There exists $p \in [2\ell-1]$ such that $v_0 v_{p+1 }, v_{2\ell} v_p \in E(G)$.
\end{claim}

\begin{proofclaim}
Let $U = N^+(v_{2\ell}) \setminus V^*$.
If $|U|\ge2$, then $v_0v_1\dots v_{2\ell}U$ is the desired antipath blow-up. Thus we may assume that $|U|\le 1$.
Since $d^+(v_{2\ell})>0$, we have $d^+(v_{2\ell})\ge k$.
Thus $d^+(v_{2\ell}, V^*) \ge k-1$.
Recall that $N^+(v_0) \subseteq V^* \setminus \{v_0\}$. In particular, $|\{v_i\colon v_{i+1}\in N^+(v_0)\}|=d^+(v_0)$.
Therefore,
\begin{align*}
	| \{v_i \colon v_{i+1} \in  N^+(v_0)  \} \cap  N^+( v_{2\ell} ) |
	& \ge d^+(v_0) + d^+(v_{2\ell}, V^*) - |V^*|\\
	& \ge k + k-1 - (2\ell+1) = 2k-2\ell - 2 \ge 2.
\end{align*}
At most one vertex in $\{v_i \colon v_{i+1} \in  N^+(v_0)  \} \cap  N^+( v_{2\ell} )$ is $v_0$. So the claim follows.
\end{proofclaim}

Let $p \in [2\ell-1]$ be such that $v_0 v_{p+1 }, v_{2\ell} v_p \in E(G)$, which exists by Claim~\ref{clm:fatend}.
Suppose that $p$ is odd.
Since $v_{p+1}v_p\in E(G)$ and $v_0v_{p+1}\in E(G)$, we have $d^+(v_{p+1}),d^-(v_{p+1})\ge k$.
Let
\begin{align*}
	W^+  = N^+(v_{p+1}) \setminus V^*
	\qquad \text{ and } \qquad
	W^-  = N^-(v_{p+1}) \setminus V^*.
\end{align*}
Note that
\begin{align*}
	|W^+| + |W^-| \ge  d^+(v_{p+1}) +  d^-(v_{p+1}) - |V^* \setminus \{v_{p+1}\}| \ge 2k - 2 \ell \ge 4.
\end{align*}
Hence at least one of $W^+$ and $W^-$ has size at least~$2$.
Therefore, one of $W^-v_{p+1}v_0\dots v_pv_{2\ell}\dots v_{p+2}$ or $v_0\dots v_pv_{2\ell}\dots v_{p+1}W^+$ is a desired antipath blow-up (see Figure~\ref{fig:fatend}).
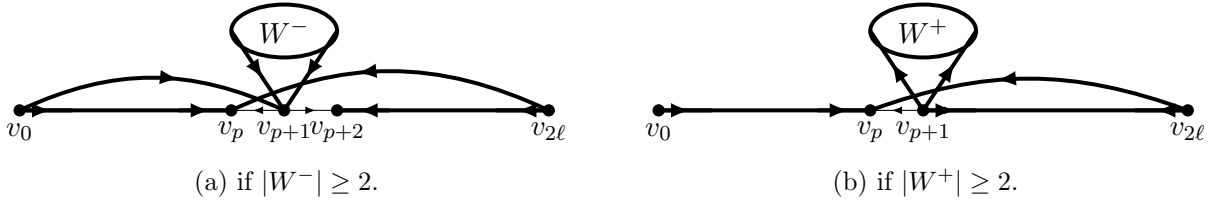
\begin{figure}
	\centering
		\begin{subfigure}{0.475\textwidth}
			\centering
		\begin{tikzpicture}[scale=0.7]

					\draw[line width =1.5pt] (6,1.5) ellipse (1cm and 0.5cm);
					\node at (6,1.5) {$W^-$};

				\foreach \x/\y in { 1/0,5/p,6/p+1,7/p+2,11/2\ell}
					{
					\filldraw[black] (\x,0) circle (3pt);
					\node at (\x,-0.4) {$v_{\y}$};
					}
				\foreach \x/\y in {6/7,6/5}
					{		\draw[->-,thin] (\x,0)--(\y,0);		}
					
				\begin{scope}[line width =1.5pt]
					\draw[->-] (5,1.5)--(6,0);
					\draw[->-] (7,1.5)--(6,0);
					\draw (1,0)--(5,0);
					\draw (7,0)--(11,0);
					\draw[bend right=25, ->-] (11,0) to (5,0);
					\draw[bend left=25, ->-] (1,0) to (6,0);
					\foreach \x/\y in {1/2,11/10,4/5,8/7}
					{		\draw[->-] (\x,0)--(\y,0);		}	
				\end{scope}
	\end{tikzpicture}
	\caption{if $|W^-|\ge 2$.}
		\label{fig:fatend:1}
	\end{subfigure}
		\begin{subfigure}{0.475\textwidth}
			\centering
		\begin{tikzpicture}[scale=0.7]

					\draw[line width =1.5pt] (6,1.5) ellipse (1cm and 0.5cm);
					\node at (6,1.5) {$W^+$};

				\foreach \x/\y in { 1/0,5/p,6/p+1,11/2\ell}
					{
					\filldraw[black] (\x,0) circle (3pt);
					\node at (\x,-0.4) {$v_{\y}$};
					}
				\foreach \x/\y in {6/5,6/7}
					{		\draw[->-,thin] (\x,0)--(\y,0);		}
					
				\begin{scope}[line width =1.5pt]
					\draw[->-] (6,0)--(5,1.5);
					\draw[->-] (6,0)--(7,1.5);
					\draw (1,0)--(5,0);
					\draw (6,0)--(11,0);
					\draw[bend right=20, ->-] (11,0) to (5,0);
					\foreach \x/\y in {1/2,11/10,6/7,4/5}
					{		\draw[->-] (\x,0)--(\y,0);		}	
				\end{scope}
	\end{tikzpicture}
	\caption{if $|W^+|\ge 2$.}
		\label{fig:fatend:2}
	\end{subfigure}
		\caption{if $p$ is odd.}
	\label{fig:fatend}
\end{figure}

Suppose that $p$ is even.
(The proof is similar, we include here for completeness.)
Let $U^+=N^+(v_{p})\setminus V^*$ and $U^-=N^-(v_{p})\setminus V^*$.
As before,
\begin{align*}
|U^+|+|U^-|\ge d^+(v_{p})+d^-(v_{p})-|V^*\setminus\{v_{p}\}|
\ge 2k-2\ell\ge4.
\end{align*}
Hence at least one of $U^+$ and $U^-$ has size at least~$2$.
Therefore one of $ U^- v_{p} v_{2\ell} \dots v_{p+1} v_0 \dots v_{p-1}$ or $v_{2 \ell} \dots v_{p+1} v_0 \dots v_p U^+$ is a desired antipath blow-up.
\end{proof}

We now outline the ideas behind the proof of Lemma~\ref{lem:twofatends}.
Let $G$ be an oriented graph with $\pd(G) \ge k$ and the longest antipath in $G$ has length $2 \ell +1 < 2k-1$.
Let $V_0 v_1 \dots v_{2\ell+1}$ be an antipath blow-up with  $|V_0|$ maximal.
We aim to find $p \in [2\ell-3]$ such that
\begin{align*}
	V_0 \subseteq N^{-}(v_{p+1}) \qquad \text{ and } \qquad v_{2\ell}v_p \in E(G).
\end{align*}
We will show that such $p$ exists by considering $i \in [\ell-1]$ with $e(V_0, \{ v_{2i},v_{2i+1} \}) + d^+(v_{2\ell},  \{ v_{2i-1},v_{2i} \} ) \ge |V_0| +2$.
We will refer to such $i$ as being  excess.
For simplicity, suppose that $p= 2i_0-1$ is odd.
Let $V' = \{v_i \colon i \in [2\ell]\}$.
As in the proof of Lemma~\ref{lem:onefatend}, we consider the number of neighbours of $v_{p+1}=v_{2i_0}$ not in~$V' \cup V_0$.
If $|N^+(v_{2i_0}) \setminus (V'\cup V_0)| \ge 2$, then we will obtain the desired antipath blow-up (see Figure~\ref{fig:podd:1}).
If $|N^-(v_{2i_0}) \setminus (V'\cup V_0)| \ge |V_0|+1$, then we find another antipath blow-up that contradicts the maximality of~$|V_0|$ (see Figure~\ref{fig:podd:2}).
In order to prove that $v_{2i_0}$ has sufficient neighbours not in~$V' \cup V_0$, we show that $v_{2i_0}$ is not adjacent to many~$v_{2i}$ such that $i$ is $V_0$-surplus.

\begin{proof}[Proof of Lemma~\ref{lem:twofatends}]
By Lemma~\ref{lem:onefatend} and considering the reverse of~$G$ if necessary, we may assume that there exists an antipath blow-up $V_0v_1\dots v_{2\ell}V_{2\ell+1}$ with $|V_0| \ge 2$.
We further assume that the antipath blow-up is chosen so that $|V_0|$ is maximum. If $|V_{2\ell+1}|\ge2$, then we are done. Hence we may assume that $V_{2\ell+1}=\{v_{2\ell+1}\}$.

Let $V' = \{v_j \colon j \in [2\ell]\}$.
For $i  \in [\ell-1]$, we say that $i$ is \emph{excess} if
\begin{align*}
	e(V_0, \{v_{2i},v_{2i+1}\})+ d^+(v_{2\ell}, \{v_{2i-1},v_{2i}\}) \ge |V_0|+2.
\end{align*}

\begin{claim} \label{clm:i0}
There exists $i \in [\ell-1]$ that is excess.
\end{claim}

\begin{proofclaim}
Let $U = N^+(v_{2\ell}) \setminus (V'\cup V_0)$.
Since $V_0v_1 \dots v_{2\ell} U$ is an antipath blow-up, we may assume that $|U| \le 1$ or else we are done.
Since $d^+(v_{2\ell})>0$, we have $d^+(v_{2\ell})\ge k$.
Thus,
\begin{align*}
d^+(v_{2\ell}, V'\cup V_0) =  d^+(v_{2\ell})-|U| \ge k-1.
\end{align*}
For $i \in [\ell-1]$, let
\begin{align*}
	 \phi(i) = e(V_0, \{v_{2i},v_{2i+1}\}) + d^+(v_{2\ell}, \{v_{2i-1},v_{2i}\}).
\end{align*}
Suppose to the contrary that the claim is false, then $\phi(i) \le |V_0|+1$ for all $i \in [\ell-1]$.
By Lemma~\ref{lem:property}\ref{itm:property:a}, $N^+(V_0) \subseteq V'$.
For all $x\in V_0$, since $xv_1\in E(G)$, we have $d^+(x)\ge k$.
Hence
\begin{align*}
	k |V_0| +  k-1 & \le \sum_{x \in V_0} d^+(x,V') + d^+(v_{2\ell}, V'\cup V_0) = e(V_0,V') + e(v_{2\ell},  V'\cup V_0)\\
	& = e(V_0,v_1) + ( e(V_0,v_{2\ell})+e(v_{2\ell},V_0)) + e(v_{2\ell}, v_{2\ell-1})  + \sum_{i \in [\ell-1]} \phi(i) \\
	&\le |V_0|+|V_0| + 1 + (\ell-1) (|V_0|+1) = (\ell+1) |V_0| +\ell,
\end{align*}
where $e(V_0,v_{2\ell})+e(v_{2\ell},V_0)\le |V_0|$ since $G$ is an oriented graph.
We obtain a contradiction as $\ell \le k-2$.
\end{proofclaim}

For $i \in [\ell-1]$, let $\phi^+(i) = e(V_0, \{v_{2i},v_{2i+1}\})$.
Let
\begin{align*}
	I &= \{i \colon \text{$i$ is $V_0$-surplus and $\phi^+(i)+\phi^+(i-1) \ge 2|V_0|+1$} \}
	\qquad \text{ and } \qquad
	I^- = \{ i-1 \colon i \in I\}.
\end{align*}
We now bound $|I|$ from below.
By Lemma~\ref{lem:property}\ref{itm:property:c} and~\ref{itm:property:e}, we have $I \subseteq [2,\ell-1]$ and $I \cap I^- = \emptyset$, respectively.
By Lemma~\ref{lem:property}\ref{itm:property:d}, we have for all $i \in I$,
\begin{align}
	\phi^+(i)+ \phi^+(i-1) =2 |V_0| +1. \label{eqn:I}
\end{align}
Let $I_0$ be the set of $i \in [\ell-1] \setminus I$ that are $V_0$-surplus.
For each $i \in I_0$, $\phi^+(i)+ \phi^+(i-1) \le 2 |V_0|$.
By Lemma~\ref{lem:property}\ref{itm:property:e}, $I \cup I_0$ does not contain two consecutive numbers.
Let $I^-_0 = \{ i-1 \colon i \in I_0\}$.
For all $i\in[\ell-1]\setminus (I\cup I^-\cup I_0\cup I_0^-)$, $i$ is not $V_0$-surplus and so $\phi^+(i)\le |V_0|$.
Therefore we have
\begin{align}
	\sum_{i \in [\ell-1]} \phi^+(i)
	& = \sum_{i \in [\ell-1] \setminus (I\cup I^- \cup I_0 \cup I_0^-)} \phi^+(i)  + \sum_{i \in I_0} \left( \phi^+(i)  +\phi^+(i-1) \right)+  \sum_{i \in I} \left( \phi^+(i)  +\phi^+(i-1) \right)\nonumber\\
	&\overset{\mathclap{\text{\eqref{eqn:I}}}}{\le}  (\ell -1 - 2|I| - 2|I_0|) |V_0|+ |I_0| \cdot 2  |V_0| +|I|(2|V_0|+1)= (\ell-1)|V_0|+|I|.
	\label{eqn:I2}
\end{align}
By Lemma~\ref{lem:property}\ref{itm:property:a}, $N^+(V_0) \subseteq V'$.
Therefore
\begin{align}
	k |V_0| & \le \sum_{x \in V_0} d^+(x) = e(V_0,V')
	\le e(V_0,v_1) + e(V_0,v_{2\ell})+ \sum_{i \in [\ell-1]} \phi^+(i)  \nonumber \\
		& \overset{\mathclap{\text{\eqref{eqn:I2}}}}{\le} 2|V_0| + (\ell-1)|V_0|+|I|= (\ell+1) |V_0| + |I| \le (k-1)|V_0| + |I|. \nonumber
\end{align}
Therefore we have
\begin{align}
	|I| \ge |V_0| \label{eqn:|I|}.
\end{align}

The next claim helps us identify the structure of an excess index.

\begin{claim} \label{clm:peven}
For all $j \in [\ell-1]$, we have $ v_{2\ell} v_{2j} \notin E(G)$ or $v_{2j+1 } \notin N^+(V_0)$.
\end{claim}

\begin{proofclaim}
Suppose to the contrary that there exist $j \in [\ell-1]$ and $x \in V_0$ such that $v_{2\ell} v_{2j}, x v_{2j+1 } \in E(G)$.
Let $W =\{x\} \cup V'$.
Note that $v_{2j} \dots v_1 x v_{2j+1} \dots v_{2\ell+1}$ is an antipath of length~$2 \ell +1$.
We claim that $N^+(v_{2j})\subseteq W$.
Let $y\in N^+(v_{2j})\setminus W$. If $y\ne v_{2\ell+1}$, then $yv_{2j}\dots v_1xv_{2j+1}\dots v_{2\ell+1}$ is an antipath of length $2\ell+2$, a contradiction.
If $y=v_{2\ell+1}$, then the edge $v_{2j}v_{2\ell+1}$ together with the antipath $v_{2j}\dots v_1xv_{2j+1}\dots v_{2\ell+1}$ forms an anticycle of length $2\ell+2$, again a contradiction by Lemma~\ref{lem:cycle}.
Thus $N^+(v_{2j})\subseteq W$.

Let $V_{2I}=\{ v_{2i} \colon i \in I \}$.
Recall that each $i \in I$ is $V_0$-surplus and $v_{2j+1}\in N^+(V_0)$.
By Lemma~\ref{lem:property}\ref{itm:property:g}, $N^+ (v_{2j}) \cap V_{2I} = \emptyset$.
Thus
\begin{align*}
	d^+(v_{2j} , W \setminus V_{2I})  = d^+(v_{2j},W) = 	d^+(v_{2j}) \ge k.
\end{align*}
Since $v_{2\ell}v_{2j}\in E(G)$ and $\pd(G)\ge k$, we have $d^-(v_{2j})\ge k$.
By Lemma~\ref{lem:property}\ref{itm:property:j}, we have
\begin{align*}
		d^- ( v_{2j},  V_{2I} ) \le 1.
\end{align*}
Let $U = N^-(v_{2j}) \setminus W$.
Note that
\begin{align*}
		|U| & = d^-(v_{2j}) - d^-(v_{2j} , W \setminus V_{2I})  - d^-(v_{2j},  V_{2I})
		 \ge k - \left( |W \setminus V_{2I}| - d^+(v_{2j} , W \setminus V_{2I})  \right)- 1
		\\
		& \ge	k-(2\ell+1 -|I| -  k ) -1 = |I|+2(k-\ell)-2 \ge |I|+2  \overset{\mathclap{\text{\eqref{eqn:|I|}}}}{\ge} |V_0|+2.
\end{align*}
Note that $U v_{2j} v_{2\ell} \dots v_{2j+1} x v_1 \dots v_{2j-1}$ is an antipath blow-up of length~$2 \ell +1$ with $|U| > |V_0|$ (see Figure~\ref{fig:peven}), contradicting the maximality of~$|V_0|$.
\begin{figure}
	\centering
		\begin{tikzpicture}[scale=0.7]

					\foreach \x/\y in { 3/x}
					{
					\filldraw[black] (\x,1.5) circle (3pt);
					\node at (\x,1.9) {${\y}$};
					}
					
					\draw[line width =1.5pt] (6,1.5) ellipse (1cm and 0.5cm);
					\node at (6,1.5) {$U$};

				\foreach \x/\y in { 1/1,5/2j-1,6/2j,7/2j+1,11/2\ell}
					{
					\filldraw[black] (\x,0) circle (3pt);
					\node at (\x,-0.4) {$v_{\y}$};
					}
				\foreach \x/\y in {6/5,6/7}
					{		\draw[->-,thin] (\x,0)--(\y,0);		}
					
				\begin{scope}[line width =1.5pt]
					\draw[->-] (5,1.5)--(6,0);
					\draw[->-] (7,1.5)--(6,0);
					\draw[->-] (3,1.5)--(1,0);
					\draw[->-] (3,1.5)--(7,0);
					\draw (1,0)--(4,0);
					\draw (7,0)--(11,0);
					\draw[bend right, ->-] (11,0) to (6,0);
					\foreach \x/\y in {2/1,8/7,11/10,4/5}
					{		\draw[->-] (\x,0)--(\y,0);		}	
				\end{scope}
	\end{tikzpicture}
	\caption{antipath blow-up considered in Claim~\ref{clm:peven}.}		
	\label{fig:peven}
\end{figure}
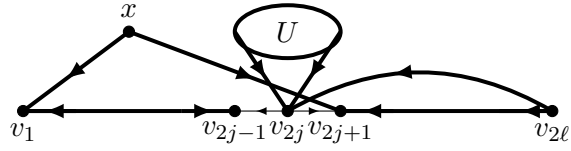
\end{proofclaim}

Fix any excess $i_0 \in [\ell-1]$, which exists by Claim~\ref{clm:i0}.

Suppose that $i_0$ is $V_0$-surplus.
By Lemma~\ref{lem:property}\ref{itm:property:c} and~\ref{itm:property:i}, there exist distinct $x_{i_0} ,x \in V_0$ and $y \in N^-(v_{2i_0})  \setminus (V' \cup V_0)$ such that $x_{i_0} v_{2 i_0},  xv_{2i_0+1} \in E(G)$.
Since $i_0$ is excess, we have $e(V_0,\{v_{2i_0},v_{2i_0+1}\})+d^+(v_{2\ell},\{v_{2i_0-1},v_{2i_0}\}) \ge |V_0|+2$.
On the other hand, Lemma~\ref{lem:property}\ref{itm:property:c} implies that $e(V_0,\{v_{2i_0},v_{2i_0+1}\})=|V_0|+1$.
Hence $d^+(v_{2\ell},\{v_{2i_0-1},v_{2i_0}\}) \ge 1$.
Since $v_{2i_0+1}\in N^+(V_0)$, Claim~\ref{clm:peven} applied with $j=i_0$ implies that $v_{2\ell}v_{2i_0}\notin E(G)$.
Therefore $v_{2\ell}v_{2i_0-1}\in E(G)$.
However, $yv_{2i_0}x_{i_0}v_1\dots v_{2i_0-1}v_{2\ell}\dots v_{2i_0+1}x$
is an antipath of length $2\ell+2$ (see Figure~\ref{fig:i_0=i}), a contradiction. Hence $i_0$ is not $V_0$-surplus.

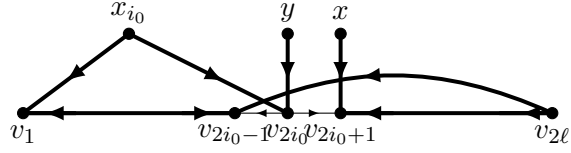
\begin{figure}
	\centering
		\begin{tikzpicture}[scale=0.7]

					\foreach \x/\y in { 3/x_{i_0}, 6/y,7/x}
					{
					\filldraw[black] (\x,1.5) circle (3pt);
					\node at (\x,1.9) {${\y}$};
					}
			
				\foreach \x/\y in { 1/1,5/2i_0-1,6/2i_0,7/2i_0+1,11/2\ell}
					{
					\filldraw[black] (\x,0) circle (3pt);
					\node at (\x,-0.4) {$v_{\y}$};
					}
				\foreach \x/\y in {6/5,6/7}
					{		\draw[->-,thin] (\x,0)--(\y,0);		}
					
				\begin{scope}[line width =1.5pt]
					\draw[->-] (7,1.5)--(7,0);
					\draw[->-] (6,1.5)--(6,0);
					\draw[->-] (3,1.5)--(1,0);
					\draw[->-] (3,1.5)--(6,0);
					\draw (7,0)--(11,0);
					\draw (1,0)--(5,0);
					\draw[bend right=24, ->-] (11,0) to (5,0);
					\foreach \x/\y in {2/1,4/5,8/7,11/10}
					{		\draw[->-] (\x,0)--(\y,0);		}	
				\end{scope}
	\end{tikzpicture}
	\caption{if $i_0$ is $V_0$-surplus.}		
	\label{fig:i_0=i}
\end{figure}

Since $i_0$ is not $V_0$-surplus, we have $e(V_0,\{v_{2i_0},v_{2i_0+1}\})\le |V_0|$.
As $i_0$ is excess and $d^+(v_{2\ell},\{v_{2i_0-1},v_{2i_0}\})\le2$,
we must have $e(V_0,\{v_{2i_0},v_{2i_0+1}\})=|V_0|$  and $d^+(v_{2\ell},\{v_{2i_0-1},v_{2i_0}\})=2$. Together with Claim~\ref{clm:peven}, we deduce that
\begin{align*}
	v_{2\ell}v_{2i_0-1}, v_{2\ell}v_{2i_0} & \in E(G), &
	v_{2i_0+1} &\notin N^+(V_0), &
	V_0 \subseteq N^-(v_{2i_0}).
\end{align*}

\begin{claim} \label{clm:podd}
There exists $i \in I$ such that $v_{2i_0}v_{2i-2} \in E(G)$ or $v_{2i-2}v_{2i_0} \in E(G)$.
\end{claim}

\begin{proofclaim}
Let $V_{2I-2} = \{ v_{2i-2} \colon i \in I\}$.
Suppose to the contrary that the claim is false,  so
\begin{align}
	(N^+(v_{2i_0}) \cup N^-(v_{2i_0}) )  \cap V_{2I-2}  = \emptyset. \label{eqn:podd}
\end{align}
Let $W^+ = N^+(v_{2 i_0 }) \setminus V'$  and $W^- = N^-(v_{2 i_0}) \setminus V'$.
Recall that $V_0 \subseteq N^-(v_{2i_0})$, so $V_0 \cap W^+ = \emptyset$.
If $|W^+| \ge 2$, then $V_0 v_1 \dots v_{2i_0-1}v_{2\ell} \dots v_{2i_0} W^+$ is a desired antipath blow-up (see Figure~\ref{fig:podd:1}).
Hence we have $|W^+| \le 1$ and so by~\eqref{eqn:podd}
\begin{align*}
d^+(v_{2i_0}, V' \setminus V_{2I-2}) = d^+(v_{2i_0}, V') \ge d^+(v_{2i_0}) - |W^+| \ge k-1.
\end{align*}
Hence
\begin{align*}
	|W^-| & = d^-(v_{2i_0}) - d^-(v_{2i_0}, V' \setminus V_{2I-2}) - d^-(v_{2i_0}, V_{2I-2})\\
	& \overset{\mathclap{\text{\eqref{eqn:podd}}}}{\ge}  k - \left( |V' \setminus V_{2I-2}| - d^+(v_{2i_0}, V' \setminus V_{2I-2}) \right) - 0\\	& \ge  k - \left( 2 \ell - |I|  - (k-1) \right)
	= |I|+2(k-\ell)-1 \ge |I|  +3  \overset{\mathclap{\text{\eqref{eqn:|I|}}}}{\ge} |V_0|+3.
\end{align*}
Let $ x \in V_0$ and $W = W^- \setminus \{x\}$, so $|W| \ge |V_0|+2$.
Then $W v_{2i_0} x v_1 \dots v_{2i_0-1} v_{2\ell } \dots v_{2i_0+1}$ is an antipath blow-up of length $2 \ell +1$ (see Figure~\ref{fig:podd:2}), contradicting the maximality of~$|V_0|$.
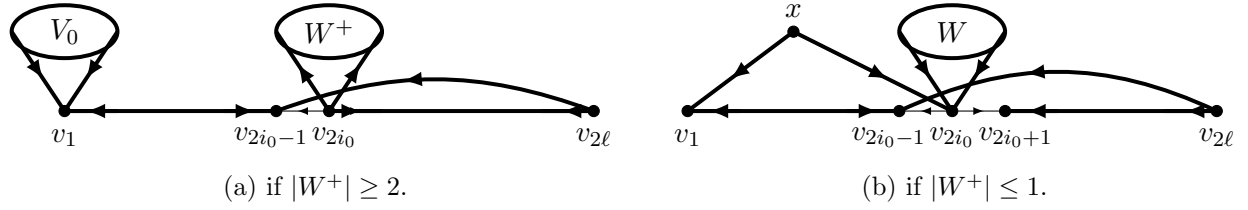
\begin{figure}
	\centering
		\begin{subfigure}{0.475\textwidth}
			\centering
		\begin{tikzpicture}[scale=0.7]

					\draw[line width =1.5pt] (6,1.5) ellipse (1cm and 0.5cm);
					\node at (6,1.5) {$W^+$};
					
					\draw[line width =1.5pt] (1,1.5) ellipse (1cm and 0.5cm);
					\node at (1,1.5) {$V_0$};

				\foreach \x/\y in { 1/1,5/2i_0-1,6/2i_0,11/2\ell}
					{
					\filldraw[black] (\x,0) circle (3pt);
					}
\node at (1,-0.5) {$v_1$};
\node at (4.9,-0.5) {$v_{2i_0-1}$};
\node at (6.1,-0.5) {$v_{2i_0}$};
\node at (11,-0.5) {$v_{2\ell}$};
				\foreach \x/\y in {6/5,6/7}
					{		\draw[->-,thin] (\x,0)--(\y,0);		}
					
				\begin{scope}[line width =1.5pt]
					\draw[->-] (6,0)--(5,1.5);
					\draw[->-] (6,0)--(7,1.5);
					\draw[->-] (0,1.5)--(1,0);
					\draw[->-] (2,1.5)--(1,0);
					\draw (1,0)--(5,0);
					\draw (6,0)--(11,0);
					\draw[bend right=20, ->-] (11,0) to (5,0);
					\foreach \x/\y in {2/1,4/5,11/10,6/7}
					{		\draw[->-] (\x,0)--(\y,0);		}	
				\end{scope}
	\end{tikzpicture}
	\caption{if $|W^+|\ge 2$.}
		\label{fig:podd:1}
	\end{subfigure}
		\begin{subfigure}{0.475\textwidth}
			\centering
		\begin{tikzpicture}[scale=0.7]

					\draw[line width =1.5pt] (6,1.5) ellipse (1cm and 0.5cm);
					\node at (6,1.5) {$W$};
					
					\foreach \x/\y in { 3/x}
					{
					\filldraw[black] (\x,1.5) circle (3pt);
					\node at (\x,1.9) {${\y}$};
					}
								
				\foreach \x/\y in { 1/1,5/2i_0-1,6/2i_0,7/2i_0+1,11/2\ell}
					{
					\filldraw[black] (\x,0) circle (3pt);
					}
\node at (1,-0.5) {$v_1$};
\node at (4.8,-0.5) {$v_{2i_0-1}$};
\node at (6,-0.5) {$v_{2i_0}$};
\node at (7.2,-0.5) {$v_{2i_0+1}$};
\node at (11,-0.5) {$v_{2\ell}$};
				\foreach \x/\y in {6/5,6/7}
					{		\draw[->-,thin] (\x,0)--(\y,0);		}
					
				\begin{scope}[line width =1.5pt]
					\draw[->-] (5,1.5)--(6,0);
					\draw[->-] (7,1.5)--(6,0);
					\draw[->-] (3,1.5)--(1,0);
					\draw[->-] (3,1.5)--(6,0);
					\draw (1,0)--(5,0);
					\draw (7,0)--(11,0);
					\draw[bend right=25, ->-] (11,0) to (5,0);
					\foreach \x/\y in {2/1,4/5,11/10,8/7}
					{		\draw[->-] (\x,0)--(\y,0);		}	
				\end{scope}
	\end{tikzpicture}
	\caption{if $|W^+|\le 1$.}
	\label{fig:podd:2}
	\end{subfigure}
		\caption{antipath blow-ups considered in Claim~\ref{clm:podd}.}
	\label{fig:podd}
\end{figure}
\end{proofclaim}

Consider any $i \in I$ with $v_{2i_0}v_{2i-2} \in E(G)$ or $v_{2i-2}v_{2i_0} \in E(G)$, which exists by Claim~\ref{clm:podd}. Since $G$ has no loops, $v_{2i-2}\ne v_{2i_0}$, so $i\neq i_0+1$. Since $i_0$ is not $V_0$-surplus while $i\in I$ is $V_0$-surplus, we have $i\ne i_0$.
Hence either $i<i_0$ or $i>i_0+1$.
Recall that $i \in I$, so $i$ is $V_0$-surplus and $\phi^+(i)+\phi^+(i-1)=2|V_0|+1$.
Lemma~\ref{lem:property}\ref{itm:property:c} and~\ref{itm:property:d} imply that there exist distinct $x_i,x \in V_0$ such that
\begin{align*}
x_i v_{2i}, x_i v_{2i+1}, x_i v_{2i-2}, x v_{2i+1}, x v_{2i-2} \in E(G).
\end{align*}
Let $U = N^-(v_{2i})  \setminus (V' \cup V_0)$.
By Lemma~\ref{lem:property}\ref{itm:property:i} and \eqref{eqn:|I|}, we have $|U| \ge |I| +3 \ge |V_0|+3$.
Recall that $v_{2i_0}v_{2i-2} \in E(G)$ or $v_{2i-2}v_{2i_0} \in E(G)$.
Then $G$ contains an antipath blow-up of length~$2 \ell +1$, namely
\begin{align*}
	& U v_{2i} x_i v_{2i+1} \dots v_{2i_0-1} v_{2\ell} \dots v_{2i_0} v_{2i-2} x v_1 \dots v_{2i-3}
		& & \text{if $ v_{2i_0} v_{2i-2}\in E(G)$ and $i <i_0$,}\\
	& U v_{2i} x_i v_{2i+1} \dots v_{2\ell}  v_{2i_0-1} \dots v_{1} x v_{2i-2} v_{2i_0} \dots v_{2i-3}
		& & \text{if $ v_{2i_0}v_{2i-2}\in E(G)$ and $i > i_0+1$,}\\
	& U v_{2i} x_i v_{1} \dots v_{2i-2} v_{2i_0} x v_{2i+1} \dots  v_{2i_0-1} v_{2\ell} \dots v_{2i_0+1}
		& & \text{if $ v_{2i-2}v_{2i_0}\in E(G)$ and  $i <i_0$,}\\
	& U v_{2i} x_i v_{2i+1} \dots v_{2\ell}  v_{2i_0-1} \dots v_1 x v_{2i_0} v_{2i-2} \dots v_{2i_0+1}
		& & \text{if $ v_{2i-2}v_{2i_0}\in E(G)$ and  $i >i_0+1$, }
\end{align*}
each of which contradicts the maximality of~$|V_0|$ (see Figure~\ref{fig:podd2}).
\begin{figure}[t]
	\centering
	\begin{subfigure}{\textwidth}
	\centering
		\begin{tikzpicture}[scale=0.8]

				\draw[line width =1.5pt] (8,1.5) ellipse (1cm and 0.5cm);
				\node at (8,1.5) {$U$};
					
				\foreach \x/\y in { 11/x_i,5/x}
					{
					\filldraw[black] (\x,1.5) circle (3pt);
					\node at (\x,1.9) {${\y}$};
					}
			\filldraw[black] (5,0) circle (3pt);
				\foreach \x/\y in { 2.5/1,5/2i-3,6/2i-2,8/2i,9/2i+1,12/2i_0,15.5/2\ell}
					{
					\filldraw[black] (\x,0) circle (3pt);
					\node at (\x,-0.4) {$v_{\y}$};
					}
					
					\filldraw[black] (7,0) circle (3pt);
					\node at (7,0.4) {$v_{2i-1}$};
					\filldraw[black] (11,0) circle (3pt);
					\node at (11,0.4) {$v_{2i_0-1}$};
					
				\foreach \x in {6,10,8}
					{		\draw[->-,thin] (\x,0)--(\x+1,0);		}
					\foreach \x in {6,8,12}
					{		\draw[->-,thin] (\x,0)--(\x-1,0);		}
					
				\begin{scope}[line width =1.5pt]
					\draw[->-] (7,1.5)--(8,0);
					\draw[->-] (9,1.5)--(8,0);
					\draw[->-] (11,1.5)--(8,0);
					\draw[->-] (11,1.5)--(9,0);
					\draw[->-] (5,1.5)--(6,0);
					\draw[->-] (5,1.5)--(2.5,0);
					\draw (2.5,0)--(5,0);
					\draw (9,0)--(11,0);
					\draw (15.5,0)--(12,0);
					\draw[bend left=23, ->-] (12,0) to (6,0);
					\draw[bend right=20, ->-] (15.5,0) to (11,0);
					\foreach \x/\y in {3.5/2.5,4/5,10/9,10/11,12/13,15.5/14.5}
					{		\draw[->-] (\x,0)--(\y,0);		}	
				\end{scope}
	\end{tikzpicture}
	\caption{if $ v_{2i_0} v_{2i-2}\in E(G)$ and $i <i_0$}
		\label{fig:podd2:1}
	\end{subfigure}
	\begin{subfigure}{\textwidth}
	\centering
		\begin{tikzpicture}[scale=0.8]

				\draw[line width =1.5pt] (12,1.5) ellipse (1cm and 0.5cm);
				\node at (12,1.5) {$U$};
					
				\foreach \x/\y in { 14/x_i,7/x}
					{
					\filldraw[black] (\x,1.5) circle (3pt);
					\node at (\x,1.9) {${\y}$};
					}
				\foreach \x/\y in { 2.5/1,5/2i_0-1,9/2i-3,10/2i-2,11/2i-1,12/2i,13/2i+1,15.5/2\ell}
					{
					\filldraw[black] (\x,0) circle (3pt);
					\node at (\x,-0.4) {$v_{\y}$};
					}

					\filldraw[black] (6,0) circle (3pt);
					\node at (6,0.4) {$v_{2i_0}$};

				\foreach \x in {6,10,8,12}
					{		\draw[->-,thin] (\x,0)--(\x+1,0);		}
					\foreach \x in {6,12,10}
					{		\draw[->-,thin] (\x,0)--(\x-1,0);		}
					
				\begin{scope}[line width =1.5pt]
					\draw[->-] (13,1.5)--(12,0);
					\draw[->-] (11,1.5)--(12,0);
					\draw[->-] (14,1.5)--(12,0);
					\draw[->-] (14,1.5)--(13,0);
					\draw[->-] (7,1.5)--(10,0);
					\draw[->-] (7,1.5)--(2.5,0);
					\draw (2.5,0)--(5,0);
					\draw (15.5,0)--(13,0);
					\draw (6,0)--(9,0);
					\draw[bend left, ->-] (6,0) to (10,0);
					\draw[bend left=20, ->-] (15.5,0) to (5,0);
					\foreach \x/\y in {3.5/2.5,4/5,6/7,8/9,14/13,15.5/14.5}
					{		\draw[->-] (\x,0)--(\y,0);		}	
				\end{scope}
	\end{tikzpicture}
	\caption{if $ v_{2i_0} v_{2i-2}\in E(G)$ and $i > i_0+1$}
		\label{fig:podd2:2}
	\end{subfigure}
	
	\begin{subfigure}{\textwidth}
	\centering
		\begin{tikzpicture}[scale=0.8]

				\draw[line width =1.5pt] (8,1.5) ellipse (1cm and 0.5cm);
				\node at (8,1.5) {$U$};
					
				\foreach \x/\y in { 11/x,6/x_i}
					{
					\filldraw[black] (\x,1.5) circle (3pt);
					\node at (\x,1.9) {${\y}$};
					}
				\foreach \x/\y in { 2.5/1,6/2i-2,8/2i,9/2i+1,12/2i_0,13/2i_0+1,15.5/2\ell}
					{
					\filldraw[black] (\x,0) circle (3pt);
					\node at (\x,-0.4) {$v_{\y}$};
					}
					
					\filldraw[black] (7,0) circle (3pt);
					\node at (7,0.3) {$v_{2i-1}$};
					\filldraw[black] (11,0) circle (3pt);
					\node at (10.8,0.3) {$v_{2i_0-1}$};

				\foreach \x in {6,10,8,12}
					{		\draw[->-,thin] (\x,0)--(\x+1,0);		}
					\foreach \x in {6,8,12}
					{		\draw[->-,thin] (\x,0)--(\x-1,0);		}
					
				\begin{scope}[line width =1.5pt]
					\draw[->-] (7,1.5)--(8,0);
					\draw[->-] (9,1.5)--(8,0);
					\draw[->-] (6,1.5)--(8,0);
					\draw[->-] (6,1.5)--(2.5,0);
					\draw[->-] (11,1.5)--(12,0);
					\draw[->-] (11,1.5)--(9,0);
					\draw (2.5,0)--(5,0);
					\draw (9,0)--(11,0);
					\draw (15.5,0)--(13,0);
					\draw[bend right=25, ->-] (6,0) to (12,0);
					\draw[bend right=22, ->-] (15.5,0) to (11,0);
					\foreach \x/\y in {3.5/2.5,6/5,10/9,10/11,14/13,15.5/14.5}
					{		\draw[->-] (\x,0)--(\y,0);		}	
				\end{scope}
	\end{tikzpicture}
	\caption{if $ v_{2i-2} v_{2i_0}\in E(G)$ and $i <i_0$}
		\label{fig:podd2:3}
	\end{subfigure}
	\begin{subfigure}{\textwidth}
	\centering
		\begin{tikzpicture}[scale=0.8]

				\draw[line width =1.5pt] (12,1.5) ellipse (1cm and 0.5cm);
				\node at (12,1.5) {$U$};
					
				\foreach \x/\y in { 14/x_i,4/x}
					{
					\filldraw[black] (\x,1.5) circle (3pt);
					\node at (\x,1.9) {${\y}$};
					}
				\foreach \x/\y in { 2.5/1,5/2i_0-1,7/2i_0+1,10/2i-2,11/2i-1,12/2i,13/2i+1,15.5/2\ell}
					{
					\filldraw[black] (\x,0) circle (3pt);
					\node at (\x,-0.4) {$v_{\y}$};
					}
					
					\filldraw[black] (6,0) circle (3pt);
					\node at (6,0.4) {$v_{2i_0}$};

				\foreach \x in {6,10,12}
					{		\draw[->-,thin] (\x,0)--(\x+1,0);		}
					\foreach \x in {6,12,10}
					{		\draw[->-,thin] (\x,0)--(\x-1,0);		}
					
				\begin{scope}[line width =1.5pt]
					\draw[->-] (13,1.5)--(12,0);
					\draw[->-] (11,1.5)--(12,0);
					\draw[->-] (14,1.5)--(12,0);
					\draw[->-] (14,1.5)--(13,0);
					\draw[->-] (4,1.5)--(6,0);
					\draw[->-] (4,1.5)--(2.5,0);
					\draw (2.5,0)--(5,0);
					\draw (15.5,0)--(13,0);
					\draw (7,0)--(10,0);
					\draw[bend right, ->-] (10,0) to (6,0);
					\draw[bend left=20, ->-] (15.5,0) to (5,0);
					\foreach \x/\y in {3.5/2.5,4/5,8/7, 10/9,14/13,15.5/14.5}
					{		\draw[->-] (\x,0)--(\y,0);		}	
				\end{scope}
	\end{tikzpicture}
	\caption{if $ v_{2i-2}v_{2i_0} \in E(G)$ and $i > i_0+1$}
		\label{fig:podd2:4}
	\end{subfigure}
	\caption{antipath blow-ups considered in the proof of Lemma~\ref{lem:twofatends}.}
	\label{fig:podd2}
\end{figure}
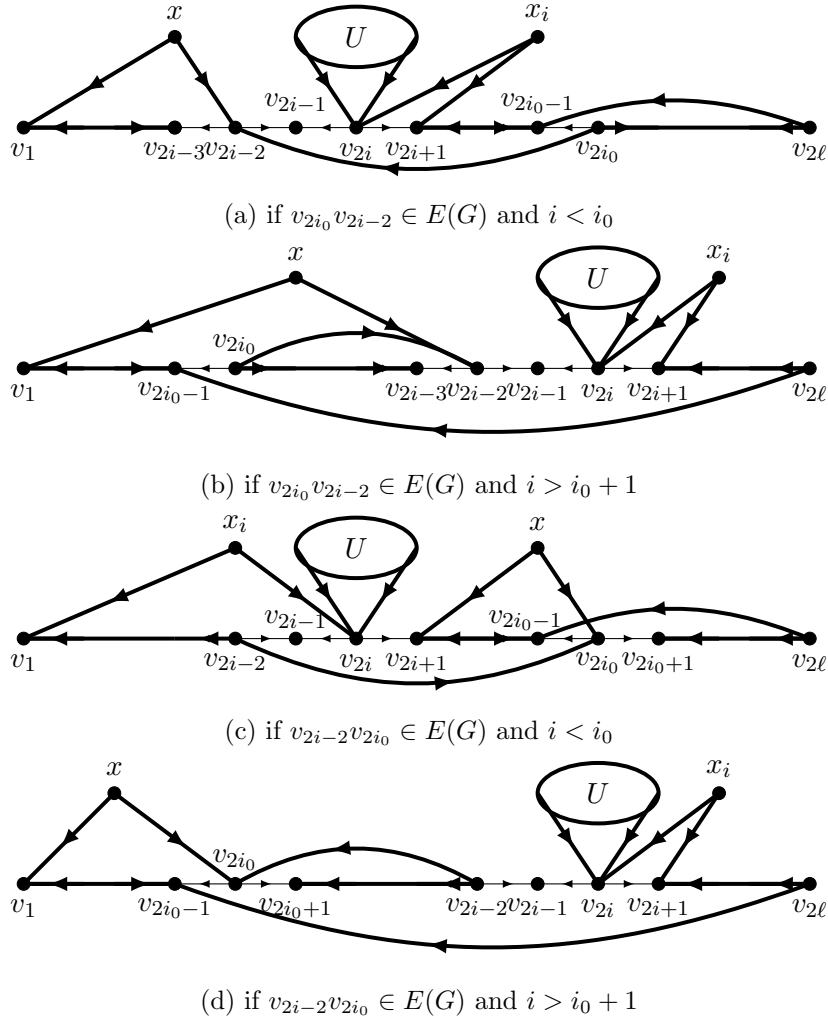

The proof of the lemma is completed.
\end{proof}

\section{Concluding remarks}

The minimum pseudo-semidegree analogue of Conjecture~\ref{conj:Stein} is false for other oriented paths.
Indeed, any oriented path that is not antidirected contains a directed path of length~$2$.
Consider a complete bipartite graph $G[A,B]$ in which every edge is directed from $A$ to $B$.
Then $\pd(G)=\min\{|A|,|B|\}$, which can be arbitrarily large, but $G$ does not contain any directed path of length~$2$.

Conjecture~\ref{conj:Stein} is now known to hold for antipaths by Theorem~\ref{thm:pseudo}, directed paths by Jackson~\cite{J1981}, and paths with one change in direction by Chen, Hou, and Zhou~\cite{CHZ2025-EJC}.
All these proofs implement some form of ``rotation and extension'' argument.
However, we are unsure how to do so for arbitrary oriented paths.

\noindent
\textbf{Data availability statement.}
There are no additional data beyond that contained within the
main manuscript.

\bibliographystyle{abbrv}
\bibliography{antipath}

\end{document}